\documentclass[11pt]{amsart}
\input epsf
\usepackage{latexsym}
\usepackage{amssymb,amsmath,amsthm,amscd, amssymb,
graphicx, enumerate, verbatim}
\usepackage[all]{xy}
\numberwithin{equation}{section}
\newtheorem{cor}[equation]{Corollary}
\newtheorem{add}[equation]{Addendum}

\newtheorem{lem}[equation]{Lemma}
\newtheorem{prop}[equation]{Proposition}
\newtheorem{thm}[equation]{Theorem}

\newtheorem{Example}[equation]{Example}

\newtheorem{remark}[equation]{Remark}
\newenvironment{rmk}{\begin{remark}\rm}{\end{remark}}
\def\co{\colon\thinspace}

\newcommand{\norm}[1]{\lVert#1\rVert}
\newcommand{\Ric}{\mbox{Ric}}
\newcommand{\Iso}{\mbox{Iso}}

\newcommand{\e}{\varepsilon}
\def\a{\alpha}
\def\G{\Gamma}
\def\g{\gamma}

\def\D{\Delta}
\def\b{\beta}

\def\d{\partial}

\def\r{\rho}

\def\l{\lambda}

\def\S1{\bf S^1}

\begin{document}

\abovedisplayskip=6pt plus3pt minus3pt
\belowdisplayskip=6pt plus3pt minus3pt

\title[Hyperplane arrangements and relative hyperbolicity]
{\bf Hyperplane arrangements in negatively curved manifolds
and relative hyperbolicity}

\rm

%20F65 Geometric group theory 
%22E40 Discrete subgroups of Lie groups
\author{Igor Belegradek\and G. Christopher Hruska}
\address{Igor Belegradek\\School of Mathematics\\ Georgia Institute of
Technology\\ Atlanta, GA 30332-0160}\email{ib@math.gatech.edu}
\address{G. Christopher Hruska\\Department of Mathematical Sciences\\
University of Wisconsin--Milwaukee\\
P.O. Box 413, Milwaukee, WI 53201-0413}\email{chruska@uwm.edu}
\thanks{Supported by NSF
\# DMS-0804038 (Belegradek) and \#DMS-0808809 (Hruska).}
\date{}
\begin{abstract} 
We show that certain aspherical manifolds
arising from hyperplane arrangements in negatively curved manifolds
have relatively hyperbolic fundamental group.
\end{abstract}
\maketitle
%\tableofcontents

\section{Introduction}

Let $M$ be a connected, complete, finite volume 
Riemannian $n$-manifold of sectional curvature
satisfying $\varkappa\le \sec(M)\le -1$ for
some constant $\varkappa$, and
let $S\subset M$ be a subset whose preimage to 
the universal cover of $M$ is the union of a 
locally finite family $\mathcal D$ of hyperplanes, 
where a {\it hyperplane} is a complete totally 
geodesic submanifold of codimension two.

In this paper we study the fundamental group of $M\setminus S$, 
which clearly surjects onto $\pi_1(M)$ as $S$ has codimension two.
One may think of $\pi_1(M\setminus S)$ as an ``overlattice'', i.e.
a group that comes with a natural surjection onto a 
lattice (cf.~\cite[page 192]{Gro-asym-inv}). 
This paper explores to what extent $\pi_1(M\setminus S)$
inherits some rigidity properties of lattices, and our approach is
to find conditions on $M, S$ implying that
$\pi_1(M\setminus S)$ is non-elementary relatively hyperbolic.
In recent years many powerful techniques have been developed 
to better understand relatively hyperbolic groups, and this
paper allows to apply the techniques to studying certain 
$M\setminus S$'s.

Set $m=\left[\frac{n}{2}\right]$. 
Following Allcock~\cite{All-JDG} we call 
$\mathcal D$ {\it normal\ \!} if hyperplanes in $\mathcal D$
are either disjoint or orthogonal, and 
if for any point $p$ lying on hyperplanes $h_1,\dots ,h_k$
in $\mathcal D$
there is a linear isomorphism of the tangent space at
$p$ onto $\mathbb R^{n-2m}\times \mathbb C^m$
that maps the tangent space to each $h_i$ to
the product of $\mathbb R^{n-2m}$
with a coordinate hyperplane in $\mathbb C^m$. 
\begin{comment}
In fact, orthogonality of hyperplanes in $\mathcal D$
and the existence of $L_p$ as above allow to
choose $L_p$ to be an isometry
of the Riemannian inner product on the tangent 
space at $p$ and the Euclidean inner product on 
$\mathbb R^n=\mathbb R^{n-2m}\times \mathbb C^m$.
\end{comment}
We call $S$ {\it normal\ \!} if $\mathcal D$ is normal.
Many examples of normal $S\subset M$ 
are known when $M$ is real hyperbolic  or complex hyperbolic.
%in particular, some beautiful examples arise from 
%Lorentzian lattices~\cite{All-JDG}.

Gromov stated in~\cite[Section 4.4A]{Gro-hyp-gr}, and Allcock
proved in~\cite{All-JDG} that if $S$ is normal in $M$,
then the metric completion of the universal Riemannian cover of
$M\setminus S$ is CAT$(-1)$, and Allcock furthermore used this to
show that the manifold $M\setminus S$ is aspherical (i.e. its
universal cover is contractible); thus most if not all topological
information about $M\setminus S$ is encoded in its fundamental group.
 
We work with the combinatorial
definition of a relatively hyperbolic group due to
Bowditch~\cite[Definition 2]{Bow-rel}, and call a subgroup of
relatively hyperbolic group 
{\it non-elementary\,} unless it is finite, virtually-$\mathbb  Z$,
or lies in a peripheral subgroup.
%If $\dim(M)=2$, then $\pi_1(M\setminus S)$
%is free non-abelian, so we only consider the case $\dim(M)>2$.
Building on ideas of Bowditch~\cite{Bow-rel}, we prove:

\begin{thm} \label{thm-intro: convex} 
Suppose that $M$ contains a closed, locally convex
subset $V$ such that $M\setminus V$ is nonempty 
and precompact in $M\!$, and $S$ lies in the interior of $V$.
If $S$ is normal in $M\!$, then 
$\pi_1(M\setminus S)$ is non-elementary relatively hyperbolic,
where peripheral subgroups are 
%exactly ?the conjugates of? 
the fundamental groups of components of $V\setminus S$;
furthermore, each component of $V\setminus S$ is aspherical, and
its inclusion into $M\setminus S$ is $\pi_1$-injective.
\end{thm}

To apply this theorem one needs to find $V$, and e.g. 
if $M$ is compact, then a natural candidate for $V$ would be (a
sufficiently small $\e$-neighborhood of) 
the smallest locally convex subset 
of $M$ that contains $S$; if this $\e$-neighborhood is a proper subset 
of $M$, then Theorem~\ref{thm-intro: convex} applies.
The same is true for non-compact $M$ except that we also need to
require that $V$ contains all cusps.

The simplest picture emerges when $S$ is a
compact smooth submanifold of $M$. 
Then $S$ is normal, and we can take $V$ 
to be the union of a small $\e$-tubular neighborhood of $S$
and sufficiently small disjoint cusp neighborhoods of $M$, which
implies the following.
%Then the components of $V\setminus S$ are homotopy equivalent to
%either circle bundles over components of $S$,
%or infranilmanifolds appearing as cusp cross-sections, and we get
%the following.

\begin{cor}\label{cor-intro: disjoint}
If $S$ is a compact smooth submanifold of $M$, then 
the group $\pi_1(M\setminus S)$ is non-elementary relatively hyperbolic,
where the peripheral subgroups are either
the fundamental groups of circle bundle
over components of $S\!$, or the fundamental groups of
compact infranilmanifolds which are cusp cross-sections
of $M$.
\end{cor}

Corollary~\ref{cor-intro: disjoint} was previously proved 
in~\cite{Bel-rh-warp, Bel-ch-warp} in the special cases when $M$
is real hyperbolic, or both $M$, $S$ are complex hyperbolic. 
The method in~\cite{Bel-rh-warp, Bel-ch-warp} involve
delicate warped product computations, which is totally different
from the approach adopted in the present article.
%, which are quite delicate in the complex hyperbolic case.  
%As in~\cite{Bel-rh-warp, Bel-ch-warp} relative hyperbolicity
%implies the following.
 
\begin{cor}\label{cor-intro: disjoint appl}
If $n\ge 3$ and $S$ is 
a compact smooth submanifold of $M$, then 
\newline
$\bullet$
$\pi_1(M\setminus S)$ has solvable word and conjugacy problems,
has finite asymptotic dimension and rapid decay property,
is co-Hopf and residually hyperbolic, 
has finite outer automorphism group; 
\newline
$\bullet$
if $\pi_1(M\setminus S)$ splits nontrivially as an amalgamated
product or HNN-extension over a subgroup $K$,
then $K$ contains a non-abelian free subgroup;
\newline
$\bullet$
if $M$ is compact, then $\pi_1(M\setminus S)$
is biautomatic and satisfies Strong Tits Alternative.
\end{cor}

As in~\cite[Theorem 1.2]{Bel-rh-warp}, one also gets
a Mostow type rigidity theorem, namely if $S_i$ is 
compact and normal in $M_i$ for $i=1,2$, and if each $M_i$ is
locally symmetric of dimension $>2$, then any 
homotopy equivalence $M_1\setminus S_1\to M_2\setminus S_2$
is homotopic to the restriction of an isometry $M_1\to M_2$ that
takes $S_1$ to $S_2$. This is only new in the ``exceptional'' case
when one of $(S_i, M_i)$'s is locally modelled on a totally real plane
in the complex hyperbolic plane. 

Most results on structure and properties of 
relatively hyperbolic groups are relative in nature, i.e.
properties of the peripheral subgroups are inherited by
the ambient relatively hyperbolic group. 
In general, the peripheral subgroups
in Theorem~\ref{thm-intro: convex} are similar to 
the group $\pi_1(M\setminus S)$ and little is known about them.
The reason relative hyperbolicity gives so much information
for compact embedded $S$ is that it is possible to choose 
$V$ such that the topology of $V\setminus S$ is easy to understand.

The main result of this paper ensures the existence of $V$ 
with an easy to understand $V\setminus S$,
provided $S$ is normal and ``sparse''.
To make this precise it helps
to introduce the following notations.
Recall that if $M$ is non-compact, then each end of $M$ has an 
arbitrary small, closed, connected neighborhood which is
called a {\it cusp}; its preimage to
the universal cover of $M$ is the union
of a family of disjoint horoballs. Fix a collection
of disjoint cusps of $M$, one for each
end, and let $Q$ be their union, and let $\mathcal H$
be the corresponding family of pairwise disjoint horoballs.
%Thus $M\setminus Q$ is nonempty and precompact in $M$.
We assume that either $S$ and $Q$ are disjoint, or $S$ intersects $\d Q$ orthogonally,
which can be always arranged by choosing $Q$ sufficiently small,
as noted in the beginning of
Section~\ref{sec: notations}.

We say that $\{S,Q\}$ is {\it $r$-sparse\ \!} if 
for any disjoint sets $A, B\in \mathcal D\cup \mathcal H$
that contain points $a\in A$ and $b\in B$ with 
$\mathrm{dist}(a,b)<r$, one has 
$a,b\in C$ for some $C\in\mathcal H$.
Since $A\cap B=\varnothing$, 
and since all horoballs in $\mathcal H$ are disjoint
it then follows that $A, B$ are hyperplanes in $\mathcal D$
that are asymptotic to the center of the horoball $C$.
In other words, $r$-sparse means that the only way two disjoint
convex sets in $\mathcal D\cup \mathcal H$ can come within
distance $r$ is inside a horoball from $\mathcal H$.
We say that $\{\mathcal D,\mathcal H\}$ is {\it $r$-sparse\ \!} 
if $\{S,Q\}$ is $r$-sparse. Compactness of $M\setminus\mathrm{Int}(Q)$ 
implies that any $\{S,Q\}$ is $r$-sparse for each sufficiently small $r$.

All this holds verbatim when $M$ is compact by letting $Q=\varnothing$.
In this case the definition of sparseness simplifies: 
$\{S,\varnothing\}$ 
is $r$-sparse if and only if 
the distance between any two disjoint hyperplanes in 
$\mathcal D$ is $\ge r$. 
We prove:

\begin{thm}\label{thm-intro: const r, cusp}  
There is a positive constant $r(\varkappa, n)$ 
such that if $S$ is normal and $\{S,Q\}$ is $r(\varkappa, n)$-sparse,
then $M\setminus S$ is diffeomorphic to the interior 
of a compact manifold $N$ such that\newline
\textup{(1)}
each component of $\d N$ is aspherical, \newline
\textup{(2)}
the inclusion $\d N\hookrightarrow N$ 
is $\pi_1$-injective,
\newline
\textup{(3)}
$\pi_1(N)$ is non-elementary relatively hyperbolic,
\newline
\textup{(4)}
conjugacy classes of 
peripheral subgroups bijectively
correspond to the fundamental groups of components of $\d N$,
\newline
\textup{(5)}
if $n\ge 3$ and $\pi_1(N)$ splits nontrivially as an amalgamated
product or HNN-extension over a subgroup $K$,
then $K$ contains a non-abelian free subgroup;
\newline
\textup{(6)}
if $n\ge 3$ and $S$ is compact, then 
$\mathrm{Out}(\pi_1(N))$ is finite;
%$\pi_1(N)$ has finite outer automorphism group, 
\newline
\textup{(7)}
if $n\ge 3$, then $\pi_1(N)$ is co-Hopf;
\newline
\textup{(8)} if $n=4$, then $\mathrm{Out}(\pi_1(N))$ is finite,
and furthermore $\pi_1(N)$ has solvable word and conjugacy problems, has
finite asymptotic dimension, and is
residually hyperbolic. 
\end{thm}

Most likely, (6) still holds
for noncompact $S$, but our proof does not apply.
In (5) we actually show that $K$ must be non-elementary
in the relatively hyperbolic group
structure given by (3)-(4). 
Part (8) hinges on various known results about 
$3$-manifold groups, appearing here
as peripheral subgroups of $\pi_1(N)$.

The manifold $N$ is constructed in a canonical way,
namely, shrinking along the rays orthogonal to $\d Q$
gives a diffeomorphism of $M\setminus S$ onto
$M\setminus (Q\cup S)$, and the latter is
the interior of a compact manifold $N$, obtained by removing
from $M$ the interior of a regular neighborhood 
of $Q\cup S$.

It is easy to construct examples to
which Theorem~\ref{thm-intro: const r, cusp} applies, e.g.
if $\pi_1(M)$ is residually finite, then any $\pi_1(M)$-invariant
locally finite normal hyperplane arrangement in the universal cover
can be thinned out by passing to a finite index subgroup 
and removing orbits of some 
hyperplanes/horoballs to ensure sparseness.

We do not know whether Theorem~~\textup{\ref{thm-intro: const r, cusp}} 
applies to any of ``natural'' examples such as
arrangements coming from Lorentzian lattices in
in~\cite{All-JDG}.
A basic difficulty is that the constant
$r(\varkappa, n)$ arising in the proof of Theorem~\ref{thm-intro: const r, cusp}
seems much larger than the sparseness constants for the `natural'' examples; we
hope to address this in future work.

To prove Theorem~\ref{thm-intro: const r, cusp}
we find $V$ as in Theorem~\ref{thm-intro: convex}
with $V\setminus (S\cup Q)$ homotopy equivalent to the boundary
of a regular neighborhood of $Q\cup S$ in $M$.
This hinges on the result of Bowditch~\cite{Bow-gf-pinch} that in a negatively pinched Hadamard manifold, any quasiconvex subset is within bounded distance to its convex hull; quasiconvexity follows from 
sparseness. This result in~\cite{Bow-gf-pinch} depends on
a delicate construction in~\cite{And}, which both contribute
to the size of $r(\varkappa, n)$.

A few properties of non-elementary relatively hyperbolic groups 
hold even if nothing is known about peripheral subgroups, e.g.
each subnormal subgroup of a non-elementary relatively hyperbolic group 
is non-elementary (as follows from a standard argument using its action 
on the ideal boundary). In particular, 
each subnormal subgroup of a non-elementary relatively hyperbolic group 
contains a non-abelian free
subgroup~\cite{Tuk} and 
has infinite dimensional second bounded cohomology~\cite{Fuj}.

In another direction it was proved in~\cite{Bel-rh-warp} that 
if a non-elementary relatively hyperbolic group is
isomorphic to a lattice in a virtually connected Lie group, 
then the lattice has real rank one. The latter cannot happen for
$\pi_1(M\setminus S)$, as long as $S$ is normal and $n>3$,
due to the following.

\begin{thm} \label{thm-intro: centralizer}
If $S$ is normal and $n>3$, then 
$\pi_1(M\setminus S)$ has a nontrivial element whose centralizer contains
a non-abelian free subgroup, and in particular,
$\pi_1(M\setminus S)$ is not 
isomorphic to a discrete isometry group of a Hadamard manifold
of pinched negative curvature.
\end{thm}

Theorem~\ref{thm-intro: centralizer} fails if $n\le 3$ in which case
$M\setminus S$ is often hyperbolizable. For normal $S$ and $n>3$, it is likely that $\pi_1(M\setminus S)$ is never isomorphic to a lattice in a virtually connected Lie group; this was proved in~\cite{ACT-orth} in 
the case when $M, S$ are complex hyperbolic, and 
Theorem~\ref{thm-intro: centralizer} proves it in case 
$\pi_1(M\setminus S)$ is non-elementary relatively hyperbolic.

It is interesting that in some cases $M\setminus S$ admits
a complete finite volume metric of $\sec\in [-1,0)$, 
see~\cite{AbrSch, Bel-rh-warp, Bel-ch-warp}, yet
%unlike relative hyperbolicity, 
the existence of such metric does not (seem to)
have significant group theoretic implications.

To prove finiteness of $\mathrm{Out}(\pi_1(M\setminus S))$ 
in Theorem~\ref{thm-intro: const r, cusp} \!(6) we
use another property of relatively hyperbolic groups that
assumes nothing about peripheral subgroups, and which is
an application of recent work of
Mineyev-Yaman~\cite{MinYam}:

\begin{prop}\label{prop: closed aspherical plus rel hyp}
Let $L$ be a closed aspherical manifold such that 
$\pi_1(L)$ is non-elementary relatively hyperbolic, 
then the simplicial volume $\norm{L}$ is positive.
\end{prop}

The proof of  Theorem~\ref{thm-intro: convex} is inspired by
arguments of Bowditch~\cite{Bow-rel}, 
where it is implicit that if $M$ contains 
a closed, locally convex subset $V$ such that $M\setminus V$ 
is nonempty and precompact in $M$, then $\pi_1(M)$ is hyperbolic 
relative to the fundamental groups of components of $V$
(see Remark~\ref{rmk: bowditch's result}). 
Combining this statement with a result of Mineyev-Yaman~\cite{MinYam}
gives the following.

\begin{thm} \label{thm-intro: simpl vol} 
Suppose that $M$ contains a closed, locally convex
subset $V$ such that $M\setminus V$ is nonempty 
and precompact in $M$. If $S\subset\mathrm{Int}(V)$, 
then the simplicial volume $\norm{M\setminus S}$ is nonzero.
\end{thm}

Simplicial volume is proper homotopy invariant 
of a manifold $U$ with values in $[0,\infty]$ 
introduced in Gromov's seminal work~\cite{Gro-vol}.
Nonvanishing of $\norm{U}$ has various geometric and topological 
consequences, e.g.
\begin{itemize}
\item
Any complete Riemannian metric on $U$
with $\Ric(U)\ge -(n-1)$ has the lower volume bound
$\norm{U}\le c_n \mathrm{Vol}(U)$~\cite[page 12]{Gro-vol}. 
\item
If $U$ is the interior of a compact manifold,
then $\norm{U}>0$ implies that $U$ admits 
no proper self-maps of degree $>1$~\cite[page 8]{Gro-vol}.
\end{itemize}
These two facts hold even if $\norm{U}$ is infinite. 
It follows from~\cite[page 59]{Gro-vol} that
$\norm{M\setminus S}$ is finite provided $M$ is compact
and $S$ is normal; this result is used in the proof of
Theorem~\ref{thm-intro: const r, cusp}\ \!(6).
By contrast, if $S$ is noncompact and 
$n=3$, then $\norm{M\setminus S}$ is always infinite, because
$M\setminus S$ is the interior of a compact manifold whose boundary
has nonzero simplicial volume. 

The structure of the paper is as follows.
In Section~\ref{sec: rel hyp} we give a general criterion
for relative hyperbolicity of a group acting on a CAT$(-1)$ space.
Section~\ref{sec: notations} is a list of
notations and standing assumptions.
In Section~\ref{sec: prelim} we establish various preliminary results
culminating in Theorem~\ref{thm: rel hyp in Section 4}, 
which is an orbifold version of
Theorem~\ref{thm-intro: convex}. Note that
Corollary~\ref{cor-intro: disjoint} is immediate from
Theorems~\ref{thm-intro: convex} and the structure of finite volume
manifolds of pinched negatively curvature.
An orbifold version of Theorem~\ref{thm-intro: const r, cusp} 
is proved in Section~\ref{sec: sparse}, while 
Section~\ref{sec: applications} contains proofs of
Corollary~\ref{cor-intro: disjoint appl},
Theorem~\ref{thm-intro: centralizer},
Proposition~\ref{prop: closed aspherical plus rel hyp},  and 
Theorem~\ref{thm-intro: simpl vol}.
In Appendix we collect some facts on $CAT(-1)$ spaces
that we could not find in the literature.

\section{Acknowledgments}
Belegradek is grateful to Ian Agol, Daniel Allcock, Greg Kuperberg, 
and Henry Wilton for helpful communications.

\section{A criterion for proving relative hyperbolicity}
\label{sec: rel hyp}

In~\cite{Bow-rel} Bowditch showed that a group
is relatively hyperbolic if and only if it 
acts on fine, connected, hyperbolic 
graph with finite quotient and finite
edge stabilizers. Here a graph is given the path-metric
in which each edge is isometric to the unit interval,
and if the path-metric is Gromov hyperbolic,
the graph is called {\it hyperbolic}. A graph is called
{\it fine\,} if for each $n$ each edge lies in only finitely many
circuits of length $<n$, where a {\it circuit} is
an embedded closed path. A family of subsets in a metric space
is called {\it $r$-separated} if the distance between any two subsets 
in the family is $\ge r$ where $r\in\mathbb R$. We prove:

\begin{thm}\label{thm: basic}
Let $X$ be a complete CAT$(-1)$ space, and let $\e>0$. Suppose
that there exists a subgroup $H\le\Iso(X)$ and a $H$-invariant
family $\mathcal A=\{ A_i\}$ 
of {$\e$-separated} closed convex subsets of $X$
such that $X\setminus\cup_i\mathrm{Int}(A_i)$ is locally 
compact, and $H$-action on 
$X\setminus\cup_i\mathrm{Int}(A_i)$ is
properly discontinuous and cocompact.
If $H_i$ is the stabilizer of $A_i$ in $H$, then
$H$ is hyperbolic relative to $\{H_i\}$.
\end{thm}

\begin{rmk} 
In this paper we apply Theorem~\ref{thm: basic}
to $X$ that is not proper. If $X$ is proper,
Theorem~\ref{thm: basic} is essentially due to Bowditch,
e.g. it follows from~\cite[7.12--7.13]{Bow-rel} and 
(elementary) Lemma~\ref{lem: bounded penetration} below.
\end{rmk}

\begin{proof} By Lemma~\ref{lem: unique segment} below
for every distinct $A_i, A_j$ there is a unique
segment $[a_{ij}, a_{ji}]$
that realizes the distance between $A_i, A_j$ where $a_{kl}\in A_k$.
%The segment $[a_{ij}, a_{ji}]$ is unique because any two segments 
%would span a flat quadrilateral in $X$, by the
%Flat Quadrilateral theorem. 

Given $u>0$, let $\Gamma_u$ be the $u$-nerve of $\{A_k\}$, i.e.
the graph with vertex set $\mathcal A$
and $A_i, A_j$ are joined by an edge if and only if 
$d(a_{ij},a_{ji})\le u$. 
Since $H$ acts cocompactly on 
$X\setminus\cup_i\mathrm{Int}(A_i)$, 
the graph $\Gamma_u$ is connected, provided $u$
is large enough. Fix such $u$ and give $\Gamma_u$ a path-metric 
with edges of length $1$.
By~\cite[Proposition 7.12]{Bow-rel} $\Gamma_u$ is a hyperbolic
metric space in this path-metric.
It is easy to see that $\G_u$ need not be fine. 
Below we replace $\G_u$ by a subgraph that is fine.

Let $\Gamma_u^\prime$ be the subgraph of $\Gamma_u$
with the same vertices in which 
$A_i, A_j$ are joined by an edge if and only if 
$[a_{ij}, a_{ji}]\cap A_k=\emptyset$ for each $k\notin\{i,j\}$.
%We refer to $\Gamma_u^\prime$ as the {\it reduced nerve}
%of $\{A_k\}$. 
The graph $\Gamma_u^\prime$ is $H$-invariant, and 
since $H$ acts properly discontinuously and cocompactly on 
$X\setminus\cup_i\mathrm{Int}(A_i)$, the quotient graph
$\Gamma_u/H$ is finite, and hence so is $\Gamma_u^\prime/H$.
Moreover, uniqueness of $[a_{ij}, a_{ji}]$
implies that edge stabilizers are finite. 

\begin{lem}
If two vertices of $\Gamma_u$ are joined by
an edge, then they are joined by a path of length $\le \frac{u}{\e}$
that lies in $\Gamma_u^\prime$. 
\end{lem}
\begin{proof}
Indeed, suppose $A_i, A_j$ are joined by an edge of $\Gamma_u$,
so that $[a_{ij}, a_{ji}]$ has length $\le u$. If $[a_{ij}, a_{ji}]$
passes through some $A_m$, then 
$\mathrm{dist}(A_i, A_j)\ge \mathrm{dist}(A_i, A_m)+
\mathrm{dist}(A_m, A_j)$.
Since $A_k$'s are $\e$-separated,
$\mathrm{dist}(A_i, A_m)$ and $\mathrm{dist}(A_m, A_j)$ are $\le u-\e$.
Repeating the process for each pair, we note that after each step
they become closer by $\e$, so the procedure terminates
at a finite sequence $A_i, \dots, A_j$
such that the sum of lengths between adjacent sets is $\le u$.
Hence $A_i$, $A_j$ are joined by a path in $\Gamma_u^\prime$
of length  $\le\frac{u}{\e}$.
\end{proof}

Thus $\Gamma_u^\prime$ is connected, and in the path-metric
induced from $\Gamma_u$ the graph
$\Gamma_u^\prime$ is quasi-isometric to $\Gamma_u$,
in particular, $\Gamma_u^\prime$ is hyperbolic.

Towards proving that $\Gamma_u^\prime$ is fine, recall that
a family of subsets $\{Q_i\}$ in a metric space is said to have 
{\it bounded penetration\,} if there is a function $D(\r)$
such that for each $k\neq j$ the intersection of the
$\r$-neighborhoods of $Q_k$, $Q_j$ has diameter $\le D(\r)$.
By Lemma~\ref{lem: bounded penetration} since the family 
$\{ A_k\}$ is $\e$-separated, it has bounded penetration.

Consider an arbitrary circuit in $\Gamma_u^\prime$
of length $n$. 
Represent it by a piecewise geodesic loop $\g$ in
$X$ written as 
$\alpha_1\cup\beta_1\cup\dots\cup\alpha_n\cup \beta_n$
where $\alpha_i$ is a geodesic segment in $A_i$ with endpoints in $\d A_i$
and $\beta_i$ is the segment $[a_i, a_{i+1}]$
joining $\d A_i$ to $\d A_{i+1}$.
Since $\{ A_k\}$ has bounded penetration, the length of $\g$
is bounded above by a linear function of $n$ as proved by Bowditch 
in comment right before Lemma 7.13 in~\cite{Bow-rel}.
(For completeness we outline his argument. That $\mathcal A$
has bounded penetration implies that
$\g$ has bounded backtracking, i.e. any two geodesic segments 
that form $\g$ only travel together for a bounded amount of time.
This is trivially true for $\b_i$'s as they have length $\le u$,
and for $\a_i$'s this follows from bounded penetration. Now
a linear bound on the length of $\g$
can be obtained from~\cite[Proposition 5.7]{Bow-rel},
which is essentially Proposition 7.3.4 of~\cite{Bow-hyperb-notes}
whose proof is fairly long.
%cited [Bo1] in~\cite{Bow-rel}
For purposes of this paper a linear bound is not important, indeed
any bound would do, and
an easier argument in~\cite[Corollary 7.2]{Bow-rel} gives a 
quadratic bound, i.e. the length of $\g$ is 
bounded above by a quadratic function of $n$.)

Now we can finish the proof that $\Gamma_u^\prime$ 
is fine. The subpath of $\g$ given as
$\g_1:=\beta_1\cup \alpha_2\cup\dots\cup \beta_n$ lies outside 
$A_1$, so the the orthogonal projection $X\to A_1$ 
maps $\g_1$ to a path $\bar\g_1$ 
in $\d A_1$ joining the endpoints of $\a_1$
(where the boundary of a subset $Z$ is defined by 
$\d Z:=\bar{Z}-\mathrm{Int}(Z)$).
Since the projection is distance-nonincreasing, the distance
between the endpoints of
$\alpha_1$ in the path-metric on $\d A_1$ induced from $X$
is bounded by the same linear function of $n$.
Applying this argument to each $\alpha_i$, we
see that $\alpha_i$ can be replaced by 
a path in $\d A_i$ with the same endpoints, giving the new loop
$\g^\prime:=\bar\g_1\cup\beta_1\cup\dots \bar\g_n\cup\beta_n$ 
in $X\setminus\cup_i\mathrm{Int}(A_i)$
whose length is bounded by a quadratic function of $n$.
The space $X\setminus\cup_i\mathrm{Int}(A_i)$, equipped with
the path metric induced from $X$, is proper by Hopf-Rinow
because it is locally compact and complete.
In this metric the family $\{\d A_k\}$ is is locally finite
because it is $\e$-separated, and hence
only finitely many $\d A_k$'s can be visited by $\g^\prime$
as above provided it contains $\beta_1$.
Thus the number of circuits of length $n$ in $\Gamma_u^\prime$
that contain a given edge is finite, so $\Gamma_u^\prime$ 
is fine. This completes the proof of Theorem~\ref{thm: basic}.
\end{proof}

%\begin{rmk} If $\{ A_k\}$ is locally finite,~\cite[Lemma 7.13]{Bow-rel} 
%%implies
%that $\Gamma_u$ is fine. Conversely, if 
%$\{ A_k\}$ is not locally finite, then for large enough $u$,
%the graph $\Gamma_u$ has infinitely many pairwise 
%adjacent vertices, so $\Gamma_u$ is not fine.
%\end{rmk}

\begin{rmk}
The proof of Theorem~\ref{thm: basic}
goes through with minor modifications when $X$ is
$\delta$-hyperbolic provided $A_i$'s are $r$-separated
with $r\gg \delta$.
\end{rmk}

\section{Notations and standing assumptions}
\label{sec: notations}

In the introduction we focused on manifolds, yet all ``natural'' 
examples we know are orbifolds, so we work equivariantly in the 
universal cover and allow lattices with torsion; this necessitates a
slight change in notations. Similarly, it would be easier
to deal with compact orbifolds but many ``natural'' examples
are non-compact.

Let $Y$ be a complete simply-connected Riemannian manifold
of sectional curvatures within $[\varkappa,-1]$, for some constant 
$\varkappa\le -1$,
and let $G$ be discrete isometry group of $Y$ such that
the orbifold $Y/G$ has finite volume. 

Let $\mathcal D$ be  
a locally finite $G$-invariant family of hyperplanes in $Y$
(recall that hyperplanes are complete totally geodesic submanifolds
of codimension two). 
Let $D$ denote the union of the hyperplanes in $\mathcal D$.

Suppose that $Y$ contains a closed, $G$-invariant, 
locally convex subset $C$ with non-empty $C^1$-smooth
boundary $\d C$ such that $D\subset\mathrm{Int}(C)$, and 
$G$ acts cocompactly on $Y\setminus \mathrm{Int}(C)$. 
Thus $\d C/G$ is a compact, and
therefore, $\d C$ has a positive normal injectivity radius.

\begin{rmk} 
\label{rmk: c1 boundary} 
All the assumptions on $C$ are crucial except for
``$\d C$ is $C^1$'' which
simplifies some matters, and causes no loss of generality. 
Indeed, if $C$ is as in the previous paragraph except 
$\d C$ is not $C^1$, then $C$ is a
(codimension zero) topological submanifold with possibly
non-smooth boundary~\cite[Theorem 1.6]{CheGro}.
%annals 1972, structure of complete manifolds 
%of nonnegative curvature 
Now according to~\cite[Theorem 1.2]{Sha} 
%Sharafutdinov, Sib. Mat. Zh, 1974, No 1
or~\cite[Lemma 5]{Poo}
%Poor, JDG, 1974, No 9.
the distance function to $C$ is $C^1$ 
near $\d C$. Since $C$ is $G$-invariant and 
$G$ acts cocompactly on $\d C$,
any sufficiently small $\e$-neighborhood of $C$ has 
$C^1$-boundary, and since the curvature is nonpositive,
the $\e$-neighborhood of $C$ is locally convex,
so by replacing $C$ with its  $\e$-neighborhood 
we may assume its boundary is $C^1$.
\end{rmk} 

The orbifold $Y/G$ is the union of a compact set, and finitely many 
{\it cusps} whose preimage in $Y$ is
a $G$-invariant family $\mathcal H$ of disjoint closed 
horoballs~\cite[Proposition 6.6]{Bow-gf-pinch}.

Suppose that every horoball in $\mathcal H$
is either disjoint from $D$, or intersects $D$ orthogonally.
(This can be always arranged by choosing horoballs in $\mathcal H$ sufficiently small, for otherwise  
since horoballs fall into finitely many $G$-equivalence types, 
there is a sequence of concentric horoballs $B_i$ 
that Hausdorff converges to their common center $z$, and a sequence of
hyperplanes $h_i$ such that the intersection $h_i\cap \d B_i$ is not orthogonal. Acting by $\mathrm{Stab}_G(z)$, 
we may assume each $h_i$ intersects a compact
fundamental domain for the 
$\mathrm{Stab}_G(z)$-action on $\d B_0$. 
By local finiteness of $\{h_i\}$, we can find $h_{i_0}$
that intersects each $B_i$, hence $z$ lies at infinity of $h_{i_0}$
implying that $h_{i_0}$ is orthogonal to the boundary of
any horoball concentric to $B_0$, which contradicts the assumption.)

Let $Y_0:=Y\setminus D$, and let $p_0\co X_0\to Y_0$ 
be the universal Riemannian covering. 
Let $X$ be the metric completion of $X_0$. 
Since $p_0$ is a local isometry, it is distance non-increasing,
so it maps Cauchy sequences to Cauchy sequences, and hence
extends to a continuous map of metric completions 
$p\co X\to Y$, which is also distance non-increasing.
Let $\D:=X\setminus X_0$.

It is proved in~\cite{All-JDG} that if $\mathcal D$ is normal,
then $X$ is CAT$(-1)$, and the inclusion $X_0\to X$ is a weak 
homotopy equivalence, in particular, $X_0$ is contractible. 

Let $\G$ be the group of all $p_0$-lifts of elements of $G$.
There
is a surjection  $\G\to G$ whose kernel is the group of automorphisms
of $p_0$. Since $G$ acts isometrically on $Y_0$, 
the group $\G$ acts isometrically
on $X_0$, and hence on $X$, and $p$ is equivariant with respect to the
surjection $\G\to G$. 
The action of $\G$ on $X_0$ is properly discontinuous, so
since $X_0$ is simply-connected, $\G$ can be identified with
the orbifold fundamental group of $X_0/\G=Y_0/G$. 
%(for if $x\in X$ is a limit of $x_i\in X_0$,
%and $\g$ is a lift of $g$, then $p(\g(x))=p(\g(x_i))=g(p(x_i))=g(p(x))$).
The sets $\D$ and $p^{-1}(C)$ are $\G$-invariant, and 
$\G$ permutes components of $p^{-1}(C)$.

%We refer to $X$ as {\it Allcock's space}; 
Some difficulties in studying this $\G$-action on
$X$ and $p^{-1}(C)$ are due to the fact that $X$ is not 
locally-compact, some points have infinite stabilizers in $\G$,
and the set of components of $p^{-1}(C)$ is not locally finite.

\begin{comment}
%\textup{(1)} 
%\newline\textup{(2)} If $[x,x^\prime]$ is a geodesic segment in $X$ such that
%$(x,x^\prime)\subset X_0$, then $p$ maps $[x,x^\prime]$ isometrically
%$to a geodesic segment in $Y$.
%\newline\textup{(3)} If $[y,y^\prime]$ is a geodesic segment in $y$ with
%$(y,y^\prime)\subset Y_0$, and if $p(u)\in (y,y^\prime)$,
%then $u$ lies in a  unique geodesic segment in $X$ that is
%projected isometrically $[y,y^\prime]$ via $p$.
\end{comment}

\section{Preliminary results}
\label{sec: prelim}

We keep notations and assumptions of Section~\ref{sec: notations}.
%As we now leave the realm of manifolds, so recall that the 
%{\it boundary of a subset $Z$\ \!} of a topological space is 
%the closure minus the interior, written $\d Z=\bar Z-\mathrm{Int}(Z)$.

\begin{lem} \label{lem: preim conv}
\textup{(1)} 
$\D=p^{-1}(D)\subset p^{-1}(\mathrm{Int}(C))=\mathrm{Int}(p^{-1}(C))$. 
\newline
\textup{(2)}
$p^{-1}(\d C)$ is the boundary of $p^{-1}(C)$. \newline
\textup{(3)} There is $r_0>0$ such that
$r_0$-neighborhood of $p^{-1}(\d C)$ is disjoint from $\D$,
and different path-components of $p^{-1}(C)$ 
have disjoint $r_0$-neighborhoods.\newline
\textup{(4)}
Path-components of $p^{-1}(C)$ coincide with connected components,
and in particular are closed in $X$.\newline
\textup{(5)}
If $A$ is a component of $p^{-1}(C)$, then $A$ is the closure
of $A\cap X_0$, and
the stabilizers of $A$ and of $A\cap X_0$ in $\G$ are equal.
\end{lem}
\begin{proof}
(1) All the unions in
\[
X_0\cup p^{-1}(D)=p^{-1}(Y_0)\cup p^{-1}(D)=p^{-1}(Y)\subset X_0\cup\D
\]
are disjoint 
%if $x\in X_0\cap p^{-1}(D)$, then $p(x)\in Y_0\cap D=\emptyset$
so $p^{-1}(D)\subset\D$. Conversely,
$\D\subset p^{-1}(D)$ else there is $x\in \D$ with $p(x)\in Y_0$,
so $x\in X_0$ contradicting $X_0\cap\D=\emptyset$, 
proving the first equality.
%
%(2) Since $p$ is a local isometry, $p$ maps $(x,x^\prime)$ to a 
%(local) geodesic, but in a Hadamard manifold local geodesic
%are distance minimizing.
%
As $D\subset\mathrm{Int}(C)$, we get 
$p^{-1}(D)\subset p^{-1}(\mathrm{Int}(C))$,
and it is trivial that
$p^{-1}(\mathrm{Int}(C))\subset \mathrm{Int}(p^{-1}(C))$. 
The last inclusion is an equality for if
$x\in \mathrm{Int}(p^{-1}(C))$, then either
$x\in \D\subset p^{-1}(\mathrm{Int}(C))$, 
or $x\in X_0$, in which case
a small neighborhood of $x$ lies in $X_0\cap p^{-1}(C)$,
hence a small neighborhood of $p(x)$ lies in $Y_0\cap \mathrm{Int}(C)$,
implying $x\in p^{-1}(\mathrm{Int}(C))$.

(2) Since $C$ is closed, so is $p^{-1}(C)$, hence the
boundary of $p^{-1}(C)$ equals to 
$p^{-1}(C)\setminus \mathrm{Int}(p^{-1}(C))$; 
the claim now follows as all unions in
\[
p^{-1}(C)=p^{-1}(\d C)\cup p^{-1}(\mathrm{Int}(C))=
p^{-1}(\d C)\cup \mathrm{Int}(p^{-1}(C)),
\] 
are disjoint, where the second equality follows from (1).

(3)  The claim follows from (2) and the 
fact that the submanifold $p^{-1}(\d C)\subset X_0$
has positive normal injectivity radius.

(4) If a connected component contains more than one path-component,
it cannot be connected by (3),
so components and path-components of $p^{-1}(C)$ coincide.
Since $C$ is closed, so is $p^{-1}(C)$.
Components of a closed set are closed (as the closure of a connected space
is connected).

(5)
The closure of $A\cap X_0$ lies in $\bar A$ which equals to 
$A$ by (4).
%Certainly, any $a\in A\cap X_0$ lies in the closure of $A\cap X_0$.
To see that any $a\in A$ lies in the closure
of $A\cap X_0$ it is enough to consider $a\in A\cap \D$,
which is a limit of some sequence
$x_i\in X_0$, where $x_i\in A$ for large $i$ as (1) implies 
$A\cap \D\subset A\cap\mathrm{Int}(p^{-1}(C))=\mathrm{Int}(A)$.

Element of $\G$ that stabilize $A$ also stabilize $A\cap X_0$
because $X_0$ is $\G$-invariant, and conversely, element of $\G$
that stabilize $A\cap X_0$, also stabilize its closure, which is $A$.
\end{proof}

\begin{rmk}
In view of (4) we refer to path-components of $p^{-1}(C)$ as
{\it components}. Lemma~\ref{lem: preim conv} also implies that
if $A$ is a component of $p^{-1}(C)$, then
$\d A=A\cap p^{-1}(\d C)$.
\end{rmk}

\begin{lem} \label{lem: univ cover}
Let $A$ be a component of $p^{-1}(C)$.
If $\mathcal D$ is normal, then \newline
\textup{(i)} $A$ is convex, and the orthogonal projection $X\to A$ maps
$X\setminus A$ onto $\d A$.\newline
%\textup{(ii)} $A$ is the closure of $A\cap X_0=A\setminus \D$.
\textup{(ii)} 
$A\cap X_0$, $p(A\cap X_0)$, $p(A)$ 
are components of 
$p^{-1}(C)\cap X_0$, $C\cap Y_0$, $C$, 
respectively. The map $p\co A\cap X_0\to p(A\cap X_0)$ is
a universal covering, and $A\cap X_0$ is contractible.
\end{lem}
\begin{proof} 
(i) Path-connected locally convex 
subsets of $CAT(0)$ spaces are 
convex~\cite[Proposition II.4.14]{BH}, so convexity of $A$
would follow from local convexity of $p^{-1}(C)$. 
Since $p\co X_0\to Y_0$ is locally isometric and $C\cap Y_0$ 
is locally convex, we know that
$p^{-1}(C\cap Y_0)=p^{-1}(C)\cap X_0$ is locally convex, so it remains
to check local convexity of $p^{-1}(C)$ at the points of 
$p^{-1}(C)\cap\D$ which by Lemma~\ref{lem: preim conv} equals to
\[
p^{-1}(C)\cap p^{-1}(D)=
p^{-1}(C\cap D)\subset p^{-1}(\mathrm{Int}(C))\subset
\mathrm{Int}(p^{-1}(C))
\]
but in $\mathrm{Int}(p^{-1}(C))$ locally convexity follows
from local convexity of $X$.

By Lemma~\ref{lem: preim conv}(4) the subset $A$ is closed,
so there is the orthogonal projection of $X$ onto $A$
that associates to a point of $X$ its (unique)
nearest point in $A$. It follows that the projection 
maps $X\setminus A$ to $\d A$.

(ii) 
It can be deduced from the proof of~\cite[Lemma 3.3]{All-JDG}
that $A\setminus\D\to A$ is a weak homotopy equivalence.
%and in fact for any compact set $K\subset A$  
%there is a homotopy $f_t\co X\to X$  
%such that $f_0=\mathrm{id}_X$, and $f_1(K)\subset X_0$. 
Since by (i) $A$ is convex, it is contractible, and hence so is
$A\setminus\D=A\cap X_0$. In fact, $A\cap X_0$ is a
path-component of $p^{-1}(C)\cap X_0$ because
any path in $p^{-1}(C)\cap X_0$ that starts in $A\cap X_0$ 
must lie in $A$. Thus the subset $A\cap X_0$ is open and closed in 
$p^{-1}(C)\cap X_0$.
Since $p_0$ is a local homeomorphism, the subset
$p(A\cap X_0)=p(A)\cap Y_0$
is open and closed in $C\cap Y_0$, and hence $p(A\cap X_0)$
is a component of $C\cap Y_0$. 
%On the other hand, $C\cap Y_0$ is path-connected
%because $C$ is path-connected and 
%$D$ is a codimension $2$ subset of $\mathrm{Int}(C)$
%with $C\setminus D=C\cap Y_0$.
Thus $p\co A\cap X_0\to p(A\cap X_0)$ is a covering map
%Massey alg top text, Lemma V.2.1
of connected manifolds, which is universal as
$A\cap X_0$ is contractible.

Finally, we show that $p(A)$ is a component of $C$.
Let $Q$ be the component of $C$ that contains $p(A)$.
Then $Q\cap Y_0$ is the component of $C\cap Y_0$
%because $Q$ is open and closed in $C$. 
containing $p(A)\cap Y_0$, so $p(A)\cap Y_0=Q\cap Y_0$.
\begin{comment}
Since $C$ is locally convex, $Q$ is convex, and
therefore the inclusion $Q\cap Y_0\to Y_0$ is $\pi_1$-injective
for if a loop in $Q\cap Y_0$ bounds a disk in $Y_0$,
the disk can be pushed into $Q\cap Y_0$
via the orthogonal projection of $Y\to Q$.
Since the covering map $A\cap X_0\to p(A)\cap Y_0$
is the restriction of $p\co X_0\to Y_0$, 
the composition of $\pi_1$-injective
maps $A\cap X_0\to p(A)\cap Y_0\hookrightarrow Y_0$
factors through the simply-connected space $X_0$, hence
$A\cap X_0$ is simply-connected.
\end{comment}
Any $y\in Q\cap D$ is the endpoint of a
geodesic segment in $Q\cap Y_0$. Lift the segment to
the cover $A\cap X_0$. The lift is isometric,
so along the lifted segment one gets a Cauchy sequence
converging to some $x\in\D$. Then $y=p(x)$, and
$x\in A$ because $A$ is closed. Thus $Q\cap D\subset p(A)$,
which together with $Q\cap Y_0=p(A)\cap Y_0$ implies
$Q\subset p(A)$, as wanted. 
\end{proof}

To state the main result of this section we let 
$\pi_1^\text{orb}$ denote the orbifold fundamental group.

\begin{thm}\label{thm: rel hyp in Section 4}
With notations and assumptions of 
Section~\ref{sec: notations}, 
if $\mathcal D$ is normal, then $\G$ is non-elementary
relatively hyperbolic. Under the identification 
$\G\cong\pi_1^\text{orb}(X_0/\G)\cong\pi_1^\text{orb}(Y_0/G)$
%of $\G$ with orbifold fundamental group of $X_0/\G=Y_0/G$
conjugacy classes of peripheral subgroups of $\G$ 
correspond to orbifold fundamental
groups of components of $(C\cap Y_0)/G$, 
considered as subgroups of $\pi_1^\text{orb}(Y_0/G)$.
\end{thm}

\begin{rmk}
It is implicit in the conclusion of the above theorems
that when $G$ acts freely on $Y_0$, the inclusion
$(C\cap Y_0)/G\hookrightarrow Y_0/G$ is $\pi_1$-injective,
and in fact it induces an isomorphism on higher homotopy groups 
because $(C\cap Y_0)/G$,  $Y_0/G$ are aspherical.
\end{rmk}
\begin{proof} We are to check that
Theorem~\ref{thm: basic} applies to the family
of components of $p^{-1}(C)$. As mentioned before,
Allcock showed that $X$  is CAT$(-1)$ when $\mathcal D$
is normal. Components of $p^{-1}(C)$ are convex by 
Lemma~\ref{lem: univ cover}\ \!(i), and $\e$-separated for some 
small positive $\e$
because $\d C$ has positive normal injectivity radius and the
tubular $\e$-neighborhood of $\d C$ lifts to a tubular $\e$-neighborhood
of $p^{-1}(\d C)$. Lemma~\ref{lem: preim conv}\ \!(1) implies that 
the complement of $\mathrm{Int}(p^{-1}(C))$ in $X$ 
is locally compact; moreover, $\G$ acts on the complement properly 
discontinuously, and cocompactly: the former again follows from
$\D\subset \mathrm{Int}(p^{-1}(C))$ and the latter holds by
identifying the quotient with $(Y\setminus\mathrm{Int}(C))/G$,
which is compact by assumption.
Thus Theorem~\ref{thm: basic} implies that $\G$ is hyperbolic relative
to the family of stabilizers of components of $p^{-1}(C)$, which 
by Lemma~\ref{lem: preim conv}\ \!(5) and  
Lemma~\ref{lem: univ cover}\ \!(ii) equal to the stabilizers 
of components of $p^{-1}(C)\cap X_0$. 

That $\G$ acts properly discontinuously on $X_0$ with quotient
$X_0/\G=Y_0/G$, and that components of $p^{-1}(C)\cap X_0$ 
are $\e$-separated easily implies that
every component of $(C\cap Y_0)/G$ is the quotient of some
component of $p^{-1}(C)\cap X_0$ by its stabilizer in $\G$.
(Indeed, fix a component $E$ of $(C\cap Y_0)/G$, and
pick a point in $p^{-1}(E)$. That point lies in some
$A\cap X_0$, which is path-connected so
$p(A\cap X_0)\subset E$, and moreover the inclusion
is equality by $\e$-separation. 
As $\G$-action permutes components 
of $p^{-1}(C)\cap X_0$, we can identify $E$ with
the quotient of $A\cap X_0$ by its stabilizer in $\G$.)

By Lemma~\ref{lem: univ cover}\ \!(ii) 
components of $p^{-1}(C)\cap X_0$ are simply-connected,
so in above notations the stabilizer of 
$A\cap X_0$ in $\G$ can be identified with 
the orbifold fundamental group of $E$.

To see that $\G$ is non-elementary first note that $\G$ is 
not virtually cyclic ($\G$ surjects onto $G$, and $G$ 
contains a non-abelian free 
subgroup, being a lattice in a negatively pinched Hadamard manifold). 
If $\G$ equals to
the stabilizer of some $A$, then $G$ stabilizes $p(A)$,
so $p(A)$ would have to contain the convex hull
of the limit set of $G$, which is $Y$, contradicting the assumption
that $C$ is a proper subset.
\end{proof}

\begin{rmk}\label{rmk: bowditch's result}
The above proof also implies that $G$ is non-elementary hyperbolic 
relative to stabilizers of components of $C$.
This result is implicit in~\cite{Bow-rel}, and it holds regardless 
of whether $\mathcal D$ is normal
by applying Theorem~\ref{thm: basic} to $Y$.
\end{rmk}

\section{Sparseness implies relative hyperbolicity}
\label{sec: sparse}

We keep notations and assumptions of Section~\ref{sec: notations}
except those involving $C$.
For each set in $\mathcal D\cup\mathcal H$
we consider its open $\r$-neighborhood, and
let $R_\r$ denote the union of these neighborhoods.
Set $R_0:=\cap_{\r>0} R_\r$, i.e. $R_0$ is the union
of all the sets in $\mathcal D\cup\mathcal H$.
%
%It is easy to show that the closure $\bar R_\delta$ of $R_\delta$ is a 
%manifold with boundary provided $\delta\in (0,\frac{r}{2})$; 
%in fact, 
We suspect that $R_\r$ is an open 
regular neighborhood of $R_0$, provided $\r<\frac{r}{2}$,
and this should be provable
with stratified Morse theory, but for our purposes 
Proposition~\ref{prop: Re to Rr is h.e.} below suffices.

\begin{prop}
\label{prop: Re to Rr is h.e.}
If $\mathcal D$ is normal, 
and $\{\mathcal D,\mathcal H\}$ is $r$-sparse, then 
for any $\r,\e$ with $0<\e< \r<\frac{r}{2}$,
the inclusion $R_\e\cap Y_0\hookrightarrow R_\r\cap Y_0$ 
is a homotopy equivalence. 
%If $G$ acts freely on $Y$, then $(R_\r\cap Y_0)/G$ 
%deformation retracts onto $(R_\r\cap Y_0)/G$.
\end{prop}
\begin{proof} 
By Whitehead's theorem 
it suffices to prove that the inclusion
induces isomorphism on all homotopy groups.
Surjectivity and injectivity will follow once
we show that every compact set $K$ (such as image of
sphere or disk) in $R_\r\cap Y_0$ can be pushed into $R_\e\cap Y_0$
by a map that is homotopic to identity of $R_\r\cap Y_0$ and 
has the property
that each point of $K\cap R_\e\cap Y_0$ stays in $R_\e\cap Y_0$
during this homotopy.
 
Note that $r$-sparseness imply that two arbitrary sets 
in $\mathcal D\cup\mathcal H$
intersect if their $\r$-neighborhoods intersect, and
normality implies that any two distinct sets  
in $\mathcal D\cup\mathcal H$ are disjoint or orthogonal, so that
if $h, h^\prime\in \mathcal D\cup\mathcal H$ intersect,
then the orthogonal projection $Y\to N_\e(h)$ takes 
$h^\prime$ to $N_\e(h)\cap h^\prime$, and also 
maps $N_\r(h^\prime)$ into itself
because the projection is distance non-increasing.

These properties ensure that for any subset
$U$ of $N_\r(h)\cap Y_0$, there exists a homotopy 
$f_{U,t}\co R_\r\cap Y_0\to R_\r\cap Y_0$ such that\newline
$\bullet$
$f_{U,0}=\mathrm{id}$ and $f_{U,1}$ maps $U$ into $N_\e(h)$, 
\newline
$\bullet$ 
tracks of $f_{U,t}$ lie on segments orthogonal to $h$, 
\newline $\bullet$ 
$f_{U,t}=\mathrm{id}$ except possibly on tracks that pass near $U$,
\newline
$\bullet$ 
if $h\in \mathcal D\cup\mathcal H$, then
$f_{U,t}$ maps $h$ and $N_\e(h)$ into themselves.

Now cover $K$ by finitely many precompact open sets 
$U_1, \dots, U_k$ such that $U_i$ lies in $N_\r(h_i)\cap Y_0$
where $h_i$'s are (not necessarily distinct) sets in 
$\mathcal D\cup\mathcal H$. Then $f_{U_1,t}$ pushes $U_1$
into $N_\e(h_1)\cap Y_0$. Set $U_i^0:=U_i$, and
let $U_i^1$ be the union of tracks of
$f_{U_1^0,t}$ passing through $U_i^0$; note that $U_i^1$ 
is precompact and 
$f_{U_2^1,t}$ pushes $U_2^1\supset U_2$ into $N_\e(h_2)\cap Y_0$. 
Let $U_i^2$ be the union of tracks of
$f_{U_2^1,t}$ passing through $U_i^1$; again $U_i^2$
is precompact, and
$f_{U_3^2,t}$ pushes $U_3^2\supset U_3$ into $N_\e(h_3)\cap Y_0$. 
Continuing in this fashion, we get homotopies
$f_{U_i^{i-1},t}$ pushing $U_i^{i-1}\supset U_i$ into $N_\e(h_i)\cap Y_0$.
The composition of these homotopies
pushes $K$ into $R_\e\cap Y_0$.   
\end{proof}

\begin{thm}\label{thm: main-orbi}
There exists constant $r>0$, depending only on $n$ and
$\varkappa$, such that if $\mathcal D$ is normal and 
$\{\mathcal D,\mathcal H\}$ is $r$-sparse,
then $\G$ is non-elementary relatively hyperbolic.
There is positive constant $\e\ll r$ such that
under the identification 
$\G\cong\pi_1^\text{orb}(X_0/\G)\cong\pi_1^\text{orb}(Y_0/G)$
%of $\G$ with the orbifold fundamental group
%of $X_0/\G=Y_0/G$ 
conjugacy classes of peripheral 
subgroups of $\G$ correspond to
orbifold fundamental groups of components 
of $(R_\e\cap Y_0)/G$,  
considered as subgroups of $\pi_1^\text{orb}(Y_0/G)$.
\end{thm}
\begin{proof}
To prove relative hyperbolicity the strategy is
to find $C$ as in Section~\ref{sec: notations}
and then apply Theorem~\ref{thm: rel hyp in Section 4}.

Since $\{\mathcal D,\mathcal H\}$ is $r$-sparse,
and sets in $\mathcal D\cup\mathcal H$ are either disjoint or orthogonal,
any two points in the same component of 
$R_0$ lie on a piecewise
geodesic with sidelengths $\ge r$ and angles at vertices $\ge \pi/2$.
Suppose $r>r_1$; then
this piecewise geodesic is a $(\l_1,\e_1)$-quasi-geodesic,
where $r_1$, $\l_1$, $\e_1$ are the constants
from by Proposition~\ref{prop:TurningAngle} corresponding to 
$\theta_1=\pi/2$. 
Any subpath of a $(\l_1,\e_1)$-quasi-geodesic,
is a $(\l_1,\e_1)$-quasi-geodesic; thus any
two points of $R_0$ can be joined by a 
$(\l_1,\e_1)$-quasi-geodesic.
By stability of quasi-geodesics~\cite[Theorem III.H.1.7]{BH},
there is a constant $r_2$ such that any 
$(\l_1,\e_1)$-quasi-geodesic in a CAT$(-1)$ space
is $r_2$-Hausdorff close
to a geodesic with the same endpoints; 
thus each component of $R_0$ is $r_2$-quasiconvex.
Bowditch~\cite[Proposition 2.5.4]{Bow-gf-pinch}
proved that each $r_2$--quasiconvex subset of a 
Hadamard manifold with $\varkappa\le \sec\le -1$ must be
$L=L(r_2,\varkappa )$-Hausdorff close to its convex hull.

For our purposes it is better to work with  
closed $\e$-neighborhoods of convex hulls, 
denoted $\mathrm{hull}_\e$, and to simplify notations we
increase $L$, replacing it with $L+\e$, 
to ensure that any $r_2$--quasiconvex subset $E$ is
$L$-Hausdorff close to 
$\mathrm{hull}_\e(E)$.

In addition to $r>r_1$, suppose $r>2L$; 
then distinct components of $R_0$
have disjoint $\mathrm{hull}_\e$'s because of $r$-sparseness.
Let $C$ be the union of $\mathrm{hull}_\e$'s
of the components of $R_0$.
Thus $C$ is a $G$-invariant, closed,
locally convex subset of $Y$, and furthermore,
each component of $C$ is $L$-Hausdorff close to 
the corresponding component of $R_0$. 
As $C$ contains the $\e$-neighborhood of $D$, we have
$D\subset\mathrm{Int}(C)$, and $G$ acts cocompactly
on $Y\setminus\mathrm{Int}(C)$ because
$C$ contains every horoball in $\mathcal H$.
If $\e$ is sufficienty small, 
Remark~\ref{rmk: c1 boundary} implies that
$\d C$ is $C^1$-smooth provided it is
non-empty.

\begin{lem} There is $\a_n\in (0,\frac{\pi}{2})$ depending only on 
$n=\dim(Y)$ such that if $2L<r\sin\a_n$, then
$C$ is a proper subset of $Y$ so that $\d C$ is non-empty. 
\end{lem}
\begin{proof}
If $G$ acts freely, the result is immediate for homological 
reasons without assuming $2L<r\sin\a_n$. Indeed, the assumption 
$r>2L$ allows to apply Proposition~\ref{prop: Re to Rr is h.e.},
which yields a deformation retraction of $R_L/G$ onto its proper subset.
Now $C\subset R_L$ so if $C=Y$, the deformation retraction
pushes the fundamental class of $Y/G$ to a proper subset,
which is impossible. This idea becomes harder
to implement in the orbifold case, 
so we settle for ad hoc argument below.

First we use $r$-sparseness to find $y\in Y$ such that
the open ball $B(y,\frac{r}{2})$
is disjoint from horoballs in $\mathcal H$, and only
intersects those hyperplanes in $\mathcal D$ that pass 
through $y$. 

Consider the hyperplanes $h_1, \dots, h_k$ with
nonempty intersection $h_0$, where we assume
$k$ is the largest possible. 
If $y\in h_0$, and
if $h\in\mathcal D$ intersects $B(y,\frac{r}{2})$,
then $h$ must intersect $h_0$. (Indeed, 
$h$ intersect each $h_i$ as the distance between 
$h$, $h_i$ is $<r$, and since $\mathcal D$ is normal,
the orthogonal projection of $h_i$ onto $h$ is $h\cap h_i$,
so the projection of $h_0$ lies in each $h\cap h_i$,
and hence lies in $h\cap h_0$,
which is therefore nonempty).
Maximality of $k$ forces $h=h_i$ for some $i$, so
$h_1, \dots, h_k$ are the only hyperplanes in $\mathcal D$
that intersect $B(y,\frac{r}{2})$. Similarly, a horoball
in $\mathcal H$ that intersects $B(y,\frac{r}{2})$ must
intersect $h_0$, so if $h_0$ is disjoint from
horoballs in $\mathcal H$, then any $y\in h_0$ has
the desired property. Note that $h_0$
is a complete totally geodesic submanifold of $Y$, and
our underlying assumptions on $\mathcal D$, $\mathcal H$
imply that if $h_0$ does intersect some $B\in \mathcal H$,
then $h_0$ is asymptotic to the center of $B$. 
If $h_0$ intersects some $B\in \mathcal H$, then we pick
$y\in h_0$ to be a point with $d(y, B)=\frac{r}{2}$.
Then the ball $B(y,\frac{r}{2})$ is disjoint from
any horoball in $\mathcal H$ because
distinct horoballs in $\mathcal H$ are $r$-separated. 

A linear algebra argument shows that
in the definition of a normal family of hyperplanes
one can choose the linear isomorphism to be isometric
with respect to the (Riemannian) inner product on the tangent 
space at $p$ and the Euclidean inner product on 
$\mathbb R^n=\mathbb R^{n-2m}\times \mathbb C^m$.
So there is $\a_n\in (0,\frac{\pi}{2})$ depending only on 
$n=\dim(Y)$, and a vector $v\in T_yY$ that forms angle $\ge\a_n$
with any hyperplane in $\mathcal D$ through $y$.
Issue a geodesic in the direction of $v$, and denote 
by $y_v$ the point where it hits $\d B(y,\frac{r}{2})$.
Let $z_v$ be a point on a hyperplane in $\mathcal D$ 
through $y$ that is closest to $y_v$. 
In the comparison triangle $\bar\D(\bar y,\bar y_v, \bar z_v)$
in $\mathbb R^2$ the angles at $\bar y$, $\bar z_v$ are $\ge \a_n$, 
$\ge \frac{\pi}{2}$ respectively, so
sine law yields 
$d(y_v, z_v)=d(\bar y_v, \bar z_v)\ge 
\frac{r}{2}\sin\a_n$. To finish the proof of the lemma
we show that $y_v\notin C$. 
Indeed, if $y_v$ were in $C$, then $y_v$ would lie 
in the $L$-neighborhood of 
some set in $\mathcal D\cup\mathcal H$. That set by above
would be a hyperplane through $y$, hence
$L\ge d(y_v, z_v)\ge \frac{r}{2}\sin\a_n$, contradicting 
the assumption $L<\frac{r}{2}\sin\a_n$.
\end{proof}

Continuing the proof of Theorem~\ref{thm: main-orbi},
we invoke Theorem~\ref{thm: rel hyp in Section 4} to conclude
that $\G$ is non-elementary relatively hyperbolic.

It remains to identify orbifold fundamental groups of
corresponding components of $(R_\e\cap Y_0)/G$ and $(C\cap Y_0)/G$. 
Construction of $C$ implies $R_0\subset R_\e\subset C\subset R_L$.
By Proposition~\ref{prop: Re to Rr is h.e.}
the inclusion $R_\e\cap Y_0\to R_L\cap Y_0$ is
a homotopy equivalence that factors through $C\cap Y_0$.
So the inclusion $C\cap Y_0\to R_L\cap Y_0$ is surjective on
homotopy groups, and moreover is injective on homotopy groups for if
$f\co S^k\to C\cap Y_0$ is null-homotopic in $R_L\cap Y_0$,
then this homotopy can be pushed to $C$ by composing it with 
the orthogonal projection of $Y$ onto the component of $C$ 
that contains $f(S^k)$. 
Thus the inclusion $R_\e\cap Y_0\to C\cap Y_0$ is a homotopy
equivalence.

Fix an arbitrary component $K$ of $C$, and set 
$K_\e:=K\cap R_\e$ and $K_0:=K\cap R_0$. 
Note that $\mathrm{Stab}_G(K\cap Y_0)=\mathrm{Stab}_G(K_\e\cap Y_0)$,
i.e. $K\cap Y_0$ and $K_\e\cap Y_0$ have equal stabilizers in $G$.
Indeed, $K$ is $\mathrm{Stab}_G(K\cap Y_0)$-invariant, and therefore 
so is $K_0$, which implies that $\mathrm{Stab}_G(K\cap Y_0)$ 
preserves $K_\e$, and hence $K_\e\cap Y_0$. On the other hand,
$\mathrm{Stab}_G(K_\e\cap Y_0)$ preserves $K_\e$, and 
hence $K_0$, which means that it preserves $K=\mathrm{hull}_\e(K_0)$,
and hence $K\cap Y_0$.

Since the inclusion $R_\e\cap Y_0\to C\cap Y_0$ is a homotopy
equivalence, so is the inclusion $K_\e\cap Y_0\hookrightarrow K\cap Y_0$, hence it lifts to a homotopy equivalence of universal covers 
$\widetilde{K_\e\cap Y_0}\hookrightarrow\widetilde{K\cap Y_0}$. 
In particular, the preimage of $K_\e\cap Y_0$ under the universal cover
$\widetilde{K\cap Y_0}\to K\cap Y_0$ is connected, and we identify it 
with $\widetilde{K_\e\cap Y_0}$. Also
Lemma~\ref{lem: univ cover} allows us identify
the universal cover $\widetilde{K\cap Y_0}\to K\cap Y_0$ 
with the restriction of
$p$ to a component of $p^{-1}(C)\cap X_0$.
With these identifications it follows that
$\mathrm{Stab}_\G(\widetilde{K\cap Y_0})=
\mathrm{Stab}_\G(\widetilde{K_\e\cap Y_0})$,
which is exactly what we claimed 
(translating to orbifold terminology).
\end{proof}

\begin{add} 
\label{add: h.e.} Under the assumptions of Theorem~\ref{thm: main-orbi},
if $G$ acts freely on $Y$, then each component of 
$(C\cap Y_0)/G$ is aspherical, and
both inclusions 
\[
(R_\e\setminus R_0)/G\hookrightarrow (R_\e\cap Y_0)/G\hookrightarrow
(C\cap Y_0)/G
\]
are homotopy equivalences. 
\end{add}
\begin{proof} 
Contracting along radial geodesics in each horoball 
in $\mathcal H$ is a $G$-equivariant deformation retraction 
$R_\r\cap Y_0\to R_\r\setminus R_0$, which descends to
a homotopy equivalence
$(R_\e\setminus R_0)/G\hookrightarrow (R_\e\cap Y_0)/G$. 
Checking that the other inclusion is a homotopy equivalence
is best done one component at a time, so let $K$ denote
a component of $C$ and use associated notations from 
the proof of Theorem~\ref{thm: main-orbi}. 
By Lemma~\ref{lem: univ cover}(ii) 
$\widetilde{K\cap Y_0}$ is contractible, so 
$K_\e\cap Y_0$, $K\cap Y_0$ are aspherical. 
The last paragraph in the proof of Theorem~\ref{thm: main-orbi}
implies that the inclusion 
\[
(K_\e\cap Y_0)/\mathrm{Stab}_G(K\cap Y_0)
\hookrightarrow (K\cap Y_0)/\mathrm{Stab}_G(K\cap Y_0)
\]
induces a $\pi_1$-isomorphism of aspherical manifolds, and hence
a homotopy equivalence.
\begin{comment}
Another way to see that
$K_\e\cap Y_0$, $K\cap Y_0$ are aspherical, is to note that
the inclusion  $K\cap Y_0\to Y_0$ induces injective maps
of homotopy groups because any disk in $Y_0$ with boundary 
in $K\cap Y_0$ can be pushed into $K\cap Y_0$ by 
orthogonal projection onto $K$.?
\end{comment}
\end{proof}

\begin{proof}[Proof of Theorem~\ref{thm-intro: const r, cusp}]
To simplify notations set $Z:= R_0/G$; this was also
denoted $Q\cup S$ in the introduction.
Fix a smooth closed regular neighborhood $T$ of $Z$ 
inside $(C\cap Y_0)/G$, and
show that the inclusion $\d T\hookrightarrow (C\cap Y_0)/G$ is a 
homotopy equivalence. 
To this end pick $\e$ is so small that $R_\e/G$ lies
inside $T$, and let $T_\e$ denote a regular neighborhood of
$Z$ inside $R_\e/G$. Then we
have inclusions 
\[
T_\e\setminus Z\underset{i}{\hookrightarrow}
(R_\e\setminus R_0)/G\underset{j}{\hookrightarrow} 
T\setminus Z\underset{k}{\hookrightarrow}
(C\cap Y_0)/G.
\]
Standard properties of regular neighborhoods (see e.g.~\cite{Bry})
imply that $j\circ i$ is a homotopy equivalence, as
$T\setminus Z$ is the union along $\d T_\e$ of
$T\setminus\mathrm{Int}(T_\e)$ and 
$T_\e\setminus Z$, which are diffeomorphic to 
$\d T_\e\times [0,1]$ and $\d T_\e\times [0,1)$, respectively.
Addendum~\ref{add: h.e.} says that
$k\circ j$ is a homotopy equivalence.
This easily implies that $j$
induces isomorphism on homotopy groups, hence
$i$, $j$, $k$ are also homotopy equivalences.
Finally, $\d T\hookrightarrow T\setminus Z\cong \d T\times [0,1)$ is
a homotopy equivalence, and hence so is
the inclusion $\d T\hookrightarrow (C\cap Y_0)/G$,
thus components of $\d T$ are aspherical and 
$\pi_1$-injectively embedded, proving (1)-(2). 
The assertions (3)--(4) are immediate from
Theorem~\ref{thm: main-orbi}. 

Parts (5) and (7) follow from~\cite[Theorem 1.3]{Bel-relhypz} 
saying that for any compact aspherical manifold
$N$ that satisfies (1)--(4), 
the group $\pi_1(N)$ is co-Hopf, and $\pi_1(N)$ does not split
as an amalgamated product or HNN-extension over subgroups
that are elementary in the relatively hyperbolic group structure 
given by (3)--(4).

To prove (6) recall that
Dru{\c{t}}u-Sapir~\cite[Theorem 1.12]{DS-split} showed that
if none of the peripheral subgroups of a relatively hyperbolic group $H$
is isomorphic to a non-elementary relatively hyperbolic group,
and if $H$ does not split over an elementary subgroup, then
$\mathrm{Out}(H)$ is finite. By the previous paragraph
$\pi_1(N)$ does not split over elementary subgroups.
Components of $\d N$ are closed aspherical manifolds, so by 
Proposition~\ref{prop: closed aspherical plus rel hyp}
if any one of them has non-elementary relatively hyperbolic 
fundamental group, then that component has positive simplicial 
volume, which we assume arguing by contradiction. Since $S$ is compact,
the assumptions in Section~\ref{sec: notations} imply that
$S$ and $Q$ are disjoint, so components of $\d N$ are either
components of $\d Q$, or the boundary components of a small
regular neighborhood of $S$. Components of $\d Q$ are
infranilmanifolds, so they have zero simplicial volume.
Thus the boundary of a regular neighborhood of $S$ has
positive simplicial volume.
Let $DM$ be the manifold obtained by 
doubling $M\setminus\mathrm{Int}(Q)$ along the boundary, 
which is $\d Q$;
we think of $S$ as sitting in one half of the double.
Then $DM\setminus S$ 
must infinite simplicial volume~\cite[pp. 17]{Gro-vol}
(cf.~\cite[pp. 56-57]{Gro-vol}),
as the interior of a compact manifold whose boundary has
positive simplicial volume. This gives a contradiction as
$DM\setminus S$ has finite simplicial
volume by~\cite[page 59]{Gro-vol} with details given in 
Lemma~\ref{lem: finte simpl vol} below. Thus we have proved (6).

For part (8) recall that each
peripheral subgroup of $\pi_1(N)$ is the fundamental group of a
closed aspherical $3$-manifold, which are geometrizable due to 
work of Thurston, Perelman and others. 
(To our knowledge the Ricci flow proof of the geometrization conjecture
has been fully written only for orientable manifolds, but non-orientable
aspherical $3$-manifolds are Haken~\cite{Hem-book}, so Thurston's
proof applies in this case. Actually, our argument below can be phrased to
depend only on the geometrization of orientable manifolds but
we felt it would only confuse the matters.)

Hempel~\cite{Hem} proved that geometrizable $3$-manifold groups are
residually finite, and in particular, they have solvable word problem.
Preaux~\cite{Pre-orie, Pre-nonorie} showed that fundamental groups
of $3$-manifold with geometrizable orientation covers
have solvable conjugacy problem.
According to Farb~\cite{Far-rel} and Bumagin~\cite{Bum-rel}, 
a relatively hyperbolic group
inherits solvability of word and conjugacy problems from its peripheral subgroups, hence these problems are solvable for $\pi_1(N)$.  
As in the proof of~\cite[Theorem 1.1(5)]{Bel-rh-warp},
the result of Osin~\cite{Osi-fil} and residual finiteness of the
peripheral subgroups implies that
$\pi_1(N)$ is (fully) residually hyperbolic. 

By the proof of (6), finiteness of $\mathrm{Out}(\pi_1(N))$ would follow
if we find a relatively hyperbolic group structure on $\pi_1(N)$
such that every peripheral subgroup is elementary in the 
relatively hyperbolic group structure given by (3)--(4),
and also is not isomorphic to a non-elementary 
relatively hyperbolic group.

According to~\cite[Corollary 1.14]{DOS}, if $H$ is hyperbolic 
relative to $\{P_i\}$ and if each $P_i$ is hyperbolic relative
to $\{P_{i}^j\}$, where we allow $P_i$ to equal $P_i^j$, then
$H$ is hyperbolic relative to $\{P_i^j\}$'s. This process
can be iterated, and in general need not terminate, but
as we note below it does terminate if we start with 
the fundamental group of a closed aspherical 3-manifold,
which would complete the proof that  $\mathrm{Out}(\pi_1(N))$
is finite. Fix a closed aspherical $3$-manifold $M$. 
By the geometrization theorem, $M$ is 
hyperbolic, Sol, Seifert fibered, or has a nontrivial JSJ decomposition 
along incompressible tori and Klein bottles whose pieces are
either hyperbolic or Seifert fibered. 
Recall that the simplicial volume $\norm{M}$ is nonzero if and only if  
$M$ is either hyperbolic or has a hyperbolic piece in the
JSJ decomposition. If $\norm{M}=0$, then 
Proposition~\ref{prop: closed aspherical plus rel hyp} implies
that $\pi_1(M)$ is not non-elementary relatively hyperbolic,
so the process terminates with $\pi_1(M)$.
If $M$ is hyperbolic, the process terminates with the trivial subgroup.
It remains to consider the case when 
there is a hyperbolic piece $\mathcal H$ in the
JSJ decomposition. Then $\pi_1(M)$  
is hyperbolic relative to fundamental groups of the components of 
$M\setminus \mathcal H$, as follows e.g. from Dahmani's combination
theorem~\cite{Dah-comb} and the fact that $\pi_1(\mathcal H)$
is hyperbolic relative to fundamental groups of tori and Klein bottles 
that lie on $\d\mathcal H$. So passing to the peripheral subgroups
corresponding to the components of $M\setminus \mathcal H$, 
and continuing in this fashion, after finitely many steps 
we end up with components that are either 
aspherical graph manifolds with incompressible boundary,
or surfaces of zero Euler characteristic appearing on the boundary
of hyperbolic pieces. In either case the process terminates; 
indeed, by Lemma~\ref{lem: graph mflds not rh} below, or
alternatively by~\cite[Theorem 11.1]{BDM}, the fundamental group
of an aspherical graph manifold with incompressible boundary 
is not non-elementary relatively hyperbolic, and the same
holds for surfaces of zero Euler characteristic because
any non-elementary relatively hyperbolic group
contains $\mathbb Z\ast\mathbb Z$, so it is not 
virtually-$\mathbb Z^2$.

Osin~\cite{Osi-asy} proved that a relatively hyperbolic group
has finite asymptotic dimension if so do all the peripheral subgroups.
Bell-Dranishnikov\cite{BelDra-asdim-trees, BelDra-asdim-ext}
showed that the class of finitely generated groups of finite asymptotic
dimension is closed under extensions, amagamated products, and
HNN-extensions. By the previous paragraph, one can 
build $\pi_1(N)$ in finitely many steps starting from the trivial group 
and using extensions, amagamated products, HNN-extensions,
and passing to peripheral subgroups, so that $\pi_1(N)$ has 
finite asymptotic dimension. This completes
the proof of (8), and hence of Theorem~\ref{thm-intro: const r, cusp}.
\end{proof}

\begin{lem}\label{lem: graph mflds not rh}
If $M$ is a compact aspherical graph manifold
such that $\d M$ is incompressible and has 
zero Euler characteristic, then $\pi_1(M)$
is not isomorphic to a non-elementary relatively hyperbolic group.
\end{lem}
\begin{proof} Arguing by contradiction suppose
$\pi_1(M)$ is non-elementary relatively hyperbolic.
Then the fundamental group of each component of $\d M$ must lie
in a peripheral subgroup, like all virtually-$\mathbb Z^2$ subgroups
do. Let $DM$ denote the double of $M$ along $\d M$.
By Dahmani's combination theorem~\cite{Dah-comb}
the relatively hyperbolic group structure on $\pi_1(M)$
defines a non-elementary relatively hyperbolic structure on $\pi_1(DM)$.
Then Proposition~\ref{prop: closed aspherical plus rel hyp} implies
$\norm{DM}>0$, which is false as 
graph manifolds have zero simplicial volume.
\end{proof}

\begin{lem}\label{lem: finte simpl vol}
If $S$ is normal and compact, then $\norm{DM\setminus S}$ is finite.
\end{lem}
\begin{proof}
A proof of finiteness of $\norm{DM\setminus S}$ is sketched
in~\cite[pp. 58--59]{Gro-vol}; 
we fill the details with help of~\cite[Theorem 5.3]{LohSau}.
%arxiv: Degree theorems and Lipschitz simplicial volume 
%for non-positively curved manifolds of finite volume.
To state Gromov's result we need two definitions. 
A subset of a space is called {\it amenable}
if for any choice of the basepoint the $\pi_1$-homomorphism induced by inclusion has amenable image. 
A sequence of subsets $U_i$ of a space $X$ is called 
{\it amenable at infinity} if there is an exhaustion of $X$
by sequence of compact sets $K_i$ such that $K_i\subset K_{i+1}$,
$U_i\subset X\setminus K_i$, and 
$U_i$ is amenable in $X\setminus K_i$ for large $i$.

A special case of 
Gromov's Finiteness Theorem~\cite[bottom of page 58]{Gro-vol}
says the following: an $n$-dimensional manifold $V$ has finite 
simplicial volume if $V$ admits 
a locally finite cover by precompact open sets $U_i$ 
that are amenable at infinity, and such that the cover $\{U_i\}$
has multiplicity $\le n$ over a subset that has precompact complement in $V$.

We extend the metric on $M\setminus Q$ arbitrarily to
a metric on $DM$.
Since $S$ is compact and normal in $DM$, we can cover 
$S$ by finitely many small metric balls
with centers on $S$ such that for each such ball $B$ 
we have $\pi_1(B\setminus S)$ is free abelian.
Since $S$ is $(n-2)$-dimensional, this cover has
a finite refinement $\{ B_j\}$ of multiplicity $\le n-1$, and we index this
cover by finite set $J$. Choose the balls sufficiently small so that
the original cover lies in some (closed) regular neighborhood $T_0$ of 
$S$ in $DM$. Also choose
a sequence of regular neighborhoods $T_k$ of $S$
all covered by $\{ B_j\}$ with $\cap_k T_k=S$ and 
$T_k\supset T_{k+1}$. 

Fix a proper function
$f\co T_0\setminus S\to\mathbb R$.
As in the proof of~\cite[Theorem 5.3]{LohSau} 
for each $j\in J$ we
find a cover $\mathcal U_j$ of $\mathbb R$ by bounded open intervals of 
multiplicity $2$ such that $\mathcal U_i\cup \mathcal U_j$ 
has multiplicity $3$ if $i\neq j$; this uses finiteness of $J$. 
Consider the cover of $T_0\setminus S $ by sets $B_j\cap f^{-1}(U)$ with
$U\in \mathcal U_j$. Each $B_j\cap f^{-1}(U)$ is relatively compact because
sets in $\mathcal U_j$ are bounded and $f$ is proper. Thus
we have covered $T_1\setminus S$ by relatively compact open subsets
$B_j\cap f^{-1}(U)$. By an elementary
argument in the proof of~\cite[Theorem 5.3]{LohSau} this cover has multiplicity $\le n$. 
It remains to show this cover is amenable at infinity, i.e. 
if $B_j\cup f^{-1}(U)$ lies in some $T_k$, then it is amenable there.
Since $T_k\hookrightarrow T_0$ is a homotopy equivalence,
it is enough to show that $B_j\cup f^{-1}(U)$ is amenable in $T_0$,
which follows as the inclusion factors through the abelian group
$\pi_1(B_j\setminus S)$.
\end{proof}

\begin{rmk}
An alternative way to prove that the boundary of a regular neighborhood
of $S$ has zero simplicial volume would be to show that it admits an 
$F$-structure, and then apply the result of
Paternain-Petean~\cite{PatPet} 
that any manifold with an $F$-structure
has zero simplicial volume. The existence of an $F$-structure,
or even the existence of local torus actions
that commute on overlaps, should follow from
normality of $S$, but we do not attempt proving it here,
as we do not see any applications.
\end{rmk}

\section{Other applications}
\label{sec: applications}

In this section we prove Corollary~\ref{cor-intro: disjoint appl},
Proposition~\ref{prop: closed aspherical plus rel hyp}, 
Theorem~\ref{thm-intro: centralizer}, and 
Theorem~\ref{thm-intro: simpl vol}.

\begin{proof}[Proof of Corollary~\ref{cor-intro: disjoint appl}]
The peripheral subgroups in this case are either 
virtually nilpotent, or have a normal infinite cyclic subgroups
with hyperbolic quotient. That $\pi_1(M\setminus S)$ are 
residually hyperbolic is provided exactly as in
the proof of~\cite[Corollary 1.4(4)]{Bel-ch-warp}. 
That $\pi_1(M\setminus S)$ is co-Hopf and has finite outer automorphism
group follows from~\cite[Theorem 1.3]{Bel-relhypz},
which also implies that $\pi_1(M\setminus S)$ does not split
nontrivially over elementary subgroups. By~\cite{Tuk}
every non-elementary subgroup contains a non-abelian free subgroup.
All the other asserted properties are proved exactly as in the proof of
in~\cite[Theorem 1.1]{Bel-rh-warp}.
%also 
%(Bogopolsky-Martino-Ventura did it for even for twisted conjugacy problem?). 
%Theorem 4.8
%ORBIT DECIDABILITY AND THE CONJUGACY PROBLEM
\end{proof}

\begin{proof}[Proof of Proposition~\ref{prop: closed aspherical plus rel hyp}]
We can assume $M$ is orientable as relative hyperbolicity is
inherited by finite index subgroups. 
Let $X$ be an Eilenberg-MacLane space for $\pi_1(L)$
in which peripheral subgroups are realized by $\pi_1$-injective
inclusions of aspherical subspaces $A_i$, one for each conjugacy class
of peripheral subgroups. Let $A=\cup_i A_i$. Peripheral subgroups
of a non-elementary relatively hyperbolic group have infinite
index, hence each $A_i$ is homotopy equivalent to an open 
$n$-manifold; hence $H_n(A_i)=0$, which implies $H_n(A)=0$. 
By the homology long exact sequence of the pair, the map
$H_n(X)\to H_n(X,A)$ is injective, i.e. the fundamental class $[L]$
is mapped to a nonzero element. By the main result of 
Mineyev-Yaman~\cite{MinYam} that element has positive
simplicial norm, and since simplicial norm is non-increasing
under continuous maps we conclude $\norm{L}>0$.
\end{proof}

Theorem~\ref{thm-intro: centralizer} is the manifold version
of Theorem~\ref{thm: centralizer} below, to for which we
adopt notations and assumptions of Section~\ref{sec: notations}
except those involving $C$.

\begin{thm} \label{thm: centralizer}
Suppose that $n>3$ and $\mathcal D$ is normal.\newline
\textup{(a)}
If $h\in \mathcal D$, 
then $h/\mathrm{Stab}_G(h)$ has finite volume.\newline
\textup{(b)}
There is a nontrivial element of $\G$ whose centralizer contains
a non-abelian free subgroup. In particular, $\G$
is not isomorphic to a discrete group of isometries of a 
negatively pinched Hadamard manifold.
\end{thm}
\begin{proof} (a) Note that $h$ 
is a negatively pinched manifold of dimension $\ge 2$.
By Section~\ref{sec: notations}, we can always arrange that
if $B\in\mathcal H$, then either $h$ and $B$ are disjoint,
or $h$ is asymptotic to the center of $B$. In the latter case
$h\cap B$ is a horoball in $h$. The set $\mathcal H_h$
of horoballs in $\mathcal H$
that intersect $h$ is $\mathrm{Stab}_G(h)$-invariant. 
Let $Y^c$, $h^c$ be the complements in $Y$, $h$
of interiors of the horoballs in $\mathcal H$,
$\mathcal H_h$, respectively. 
Since $Y^c /G$ is compact, local finiteness 
of $\mathcal D$ implies that $h^c/\mathrm{Stab}_G(h)$ is a
compact subset of $Y^c /G$. 
(Otherwise, there is a sequence of points in $h^c$, whose 
projections
in $h^c/\mathrm{Stab}_G(h)$ converge to a point outside  
$h^c/\mathrm{Stab}_G(h)$, and hence the points have lifts lying
on $G$-images of $h$ and converging to a point of $Y$, contradicting
local finiteness).
Thus for each $B\in\mathcal H_h$,
the subgroup of $\mathrm{Stab}_G(h)$ that preserves the
horoball $h\cap B$ acts  
cocompactly in $\d B\cap h$, so $h/\mathrm{Stab}_G(h)$ has finite volume.

(b)
First suppose that there are $h, h_1\in\mathcal D$ that intersect.
Since $h/\mathrm{Stab}_G(h)$ has finite volume, 
there is $g\in\mathrm{Stab}_G(h)$
that moves $h\cap h_1$ to some disjoint hyperplane in $h$,
e.g. let $g$ be hyperbolic element whose attracting limit point
is not at infinity of $h\cap h_1$, so that powers of $g$
bring $h\cap h_1$ within an arbitrary small neighborhood
of the attracting point. Since $h_1$ is orthogonal to $h$, so
is $h_2:=g(h_1)$, and then $h_1$, $h_2$ must be disjoint. 
Since $\mathcal D$ is normal, 
$\pi_1(Y\setminus \{h,h_1,h_2\})$ is the amalgamated
product of $\pi_1(Y\setminus \{h, h_1\})\cong\mathbb Z^2$ 
and $\pi_1(Y\setminus \{h, h_2\})\cong\mathbb Z^2$ along 
$\pi_1(Y\setminus \{h\})\cong\mathbb Z$. 
The obvious surjection $\pi_1(Y\setminus D)\to 
\pi_1(Y\setminus \{h,h_1,h_2\})$ splits.
(Indeed, take a smooth path in $h$ that joins $h$ to $h_1$
and intersects no other hyperplane in $\mathcal D$. 
The path has a small neighborhood $V$  
that intersects no hyperplane in $\mathcal D\setminus \{h,h_1,h_2\}$,
and has the property that the inclusion-induced map
$\pi_1(V\setminus \{h,h_1,h_2\})\to \pi_1(Y\setminus \{h,h_1,h_2\})$
is an isomorphism, hence the inclusion-induced map
$\pi_1(V\setminus \{h,h_1,h_2\})\to \pi_1(Y\setminus\mathcal D)$
gives rise to a splitting.)
Since $\pi_1(Y\setminus D)$ is a
subgroup of $\G$, so is $\pi_1(Y\setminus \{h,h_1,h_2\})$,
which by above is isomorphic to 
$\mathbb Z\times (\mathbb Z\ast\mathbb Z)$, and hence
the centralizer of 
an element of $\G$ contains $\mathbb Z\ast\mathbb Z$.

It remains to study the case when no two hyperplanes in $\mathcal D$
intersect. Fix any $h\in\mathcal D$. Then $\mathrm{Stab}_G(h)$ acts on
$Y\setminus h$, hence an index two subgroup of $\mathrm{Stab}_G(h)$
acts trivially on $\pi_1(Y\setminus h)\cong\mathbb Z$.
By (a) $h/\mathrm{Stab}_G(h)$ has finite volume, so some
$g_1, g_2\in \mathrm{Stab}_G(h)$ generate a free subgroup.
Consider $\g\in\pi_1(Y\setminus D)\subset\G$ generated
by a loop around $h$. The loop is preserved by $g_i$
up to free homotopy, hence there is a lift $\g_i\in\G$ of $g_i$ 
that commutes with $\g$. Since $g_1, g_2$ generate a free subgroup,
so do $\g_1, \g_2$, as promised. 

Finally, in a discrete group of isometries of a 
negatively pinched Hadamard manifold the centralizer $C(g)$
of any infinite order element $g$ is virtually nilpotent. 
(Indeed, $C(g)$ preserves the limit set of $g$, which consists
of one or two points. In the former case, $C(g)$
is virtually nilpotent by~\cite{Bow-parab}, and in the latter
case $C(g)$ is virtually-$\mathbb Z$ because it acts properly 
discontinuously on the geodesic line joining the limit points).
\end{proof}

\begin{rmk}
Part (a) supplies an elementary proof for a claim made
in~\cite{ACT-orth} and proved later in~\cite{ACT-add}.
\end{rmk}

The proposition below is immediate from 
the main result of~\cite{MinYam}.

\begin{prop} 
\label{prop: excision}
For $n>1$, let $N$ be a compact aspherical $n$-manifold
with boundary. Then $\norm{N,\d N}> 0$ if
$N$ contains a codimension zero compact submanifold $U$ with 
$\d N\subset\mathrm {Int} (U)$ and such that\newline
$\bullet$ 
every component $U_i$ of $U$ is aspherical and $\pi_1$-incompressible in $N$,\newline
$\bullet$ 
$\pi_1(N)$ is non-elementary hyperbolic relative to $\{\pi_1(U_i)\}$.
\end{prop}
\begin{proof} Suppose first that $N$ is orientable.
The submanifold $L:=N\setminus\mathrm {Int}(U)$ is compact,
so the fundamental class of $[L, \d L]$ generates 
$H_n(L,\d L)\cong\mathbb Z$ which equals to $H_n(N,U)$ by excision.
But in 
$H_n(N,U)$ the classes $[N,\d N]$, $[L, \d L]$ are equal,
so $[N,\d N]$ is nonzero in $H_n(N,U)$. 
By~\cite{MinYam} nonzero classes in $H_k(N,U)$, $k>1$ 
have positive simplicial norm; this applies to
$[N,\d N]\in H_k(N,U)$. Since simplicial norm is nonincreasing
under continuous maps, $[N,\d N]$ has positive simplicial
norm in $H_n(N, \d N)$, as claimed. 
The case when $N$ is non-orientable follows
by working in the oriented $2$-fold covers because 
codimension zero embeddings of manifold are orientation-true,
and relative hyperbolicity is inherited by finite index subgroups.
%(whose peripheral subgroups are obtained by intersections.)
\end{proof}

\begin{proof}[Proof of Theorem~\ref{thm-intro: simpl vol}]
Fix $Q\subset \mathrm{Int}(V)$ 
as in the introduction, and pick
closed regular neighborhood $T_1$, $T_2$ of $S\cup Q$ in $M$
so small that \[
T_1\subset \mathrm{Int}(T_2)\subset\mathrm{Int}(V).
\]
Proposition~\ref{prop: excision} applies 
to $N:=M\setminus \mathrm{Int}(T_1)$ and
$U:=V\setminus \mathrm{Int}(T_2)$, therefore $\norm{N,\d N}>0$.
It is immediate from definitions that
$\norm{\mathrm{Int}(N)}\ge\norm{N, d N}$, 
see~\cite[pages 17, 58]{Gro-vol}, and the claimed 
result follows as $\mathrm{Int}(N)$ is diffeomorphic to 
$M\setminus S$. 
\end{proof}

\appendix

\section{Some facts about CAT$(-1)$ spaces}
\label{app: cat(-1)}

Results of this appendix are well-known but we could not 
find references.

\begin{lem}
\label{lem: unique segment}
If $A, B$ are closed convex
subsets of a complete CAT$(-1)$ space $X$
such that $D:=\mathrm{dist}(A,B)>0$, then there is 
a unique geodesic segment with endpoints in $A, B$
of length $D$.
\end{lem}
\begin{proof} 
Let $b\co X\to B$ denote the orthogonal projection. 
%which makes sense as $B$ is closed convex
To prove existence, take $x_i\in A$ with 
$\mathrm{dist}(x_i, B)\in [D, D+\frac{1}{i})$,
look at the triangles $\D(x_i, x_{j}, b(x_i))$,
$\D(x_{j}, b(x_{j}), b(x_i))$, and lay the corresponding
comparison triangles $\D(\bar x_i, \bar x_{j}, \overline{b(x_i)})$,
$\D(\overline{b(x_{j})}, \bar x_{j},  \overline{b(x_i)})$ in 
the hyperbolic plane $\bf H^2$ next to each other along the side 
$[ \bar x_{j},\overline{b(x_i)}]$.
The distance between the sides
$[\bar x_i, \bar x_{j}]$, $[\overline{b(x_i)}, \overline{b(x_{j})}]$,
is $\ge D$ for if they could be joined by a segment of length $<D$,
then the parts of the segment lying on different
sides of $[ \bar x_{j},\overline{b(x_i)}]$ would define, by comparison,
an even shorter path concatenated from two geodesic segments and
joining $[x_i, x_{j}]\subset A$ with $[b(x_i), b(x_{j})]\subset B$.
Therefore $[\bar x_i, \bar x_{j}]$ lies outside the 
$D$-neighborhood of $[\overline{b(x_i)}, \overline{b(x_{j})}]$, 
and also inside 
$D+\frac{1}{i}$-neighborhood of $[\overline{b(x_i)}, \overline{b(x_{j})}]$
because the neighborhood is convex and contains the endpoints 
of $[\bar x_i, \bar x_{j}]$. 

We now prove that $\{x_i\}$ is Cauchy. Arguing by contradiction
assume that $d(x_i, x_{j(i)})$ is bounded away from zero
for some subsequence $j=j(i)$. Applying isometries of $\bf H^2$
we may assume that $\overline{b(x_i)}$ is independent of $i$, so
as $i\to\infty$ the segments $[\overline{b(x_i)}, \overline{b(x_{j(i)})}]$
subconverge to a segment $S$ (possibly of zero or infinite length), and hence
$[\bar x_i, \bar x_{j(i)}]$ subconverge to a 
geodesic segment on the boundary of the $D$-neighborhood of $S$,
which by elementary hyperbolic geometry 
contains no positive length geodesic segments. 
Thus the limit of $[\bar x_i, \bar x_{j(i)}]$ on the boundary
of the $D$-neighborhood of $S$ is a point, and
hence $d(x_i, x_{j(i)})=d(\bar x_i, \bar x_{j(i)})\to 0$.

Since $X$ is complete and $A$ is closed, $\{x_i\}$ has a limit $x\in A$, and
$[x, b(x)]$ is a segment of length $D$ joining $A, B$. 
Uniqueness also follows 
for if $[x, b(x)]$, $[y, b(y)]$ are length $D$
segments joining $A, B$, then
$[\bar x, \bar y]$ lies on the boundary of the $D$-neighborhood
of $[\overline{b(x)}, \overline{b(y)}]$, forcing $\bar x=\bar y$
and hence $x=y$.
\end{proof}

\begin{rmk}
Existence of a shortest geodesic joining $A,B$ fails if $D=0$ 
(for asymptotic geodesics in $\bf H^2$), or if
$D>0$ but $X=\mathbb R^2$ (for the subsets $y\ge e^x$ and 
$y\le -1$).
\end{rmk}

We refer to Section~\ref{sec: rel hyp} for 
definition of $r$-separation and bounded penetration.

\begin{lem}\label{lem: bounded penetration}
Any family of $r$-separated convex subspaces 
of a CAT$(-1)$ space has bounded penetration.
\end{lem}
\begin{proof}
Otherwise there is $D$ such that for each $i$ there are
$r$-separated convex subsets $C_i$, $B_i$ in the family
and points $c_i, c_i^\prime\in C_i$, $b_i, b_i^\prime\in B_i$
such that $d(c_i, b_i)+d(c_i^\prime, b_i^\prime)<D$, 
while $d(c_i, c_i^\prime)>i$, $d(b_i, b_i^\prime)>i$.
By comparison with the hyperbolic plane there is $\e_i\to 0$
such that the midpoint of $[c_i, c_i^\prime]\subset C_i$ is 
$\e_i$-close to a point of $[c_i, b_i^\prime]$ 
which in turn is $\e_i$-close to a point of 
$[b_i, b_i^\prime]\subset B$, and we get a contradiction
for $i$ with $2\e_i<r$. 
%which proves Lemma~\ref{lem: bounded penetration}.
\end{proof}

The following generalizes~\cite[Lemma
11.3.4]{ECHLPT} on piecewise geodesics in the real hyperbolic space.

\begin{prop}
\label{prop:TurningAngle}
For each angle $\theta_1>0$ there are constants $r_1$, $\l_1$,
$\e_1$ such that 
any piecewise geodesic in a CAT$(-1)$ space whose
geodesic pieces have length
$\ge r_1$ and such that successive pieces meet at an angle 
$\ge \theta_1$ is a $(\l_1,\e_1)$--quasi-geodesic.
\end{prop}

\begin{proof} 
We are to apply~\cite[Theorem~3.1.4]{CDP} that
for any constants $\delta$, $\lambda_0$, $\epsilon_0$ 
there are constants $\lambda_1$, $\epsilon_1$, $r_1$
such any $r_1$--local $(\lambda_0,\epsilon_0)$--quasi-geodesic
in a $\delta$--hyperbolic space is
a global $(\l_1,\e_1)$--quasi-geodesic.

Fix $\delta$ such that 
any CAT$(-1)$ space $\delta$--hyperbolic~\cite[Proposition~III.H.1.2]{BH}. 
Let $c_1$, $c_2$ be two geodesic rays in the hyperbolic plane emanating
from a common point $p$ at angle $\theta$.
By an elementary computation 
%using the Hyperbolic Law of Cosines 
$f(\theta,t):= d(c_1(t),c_2(t))$ is increasing in $\theta$ 
for each fixed $t>0$, so let
$\r=\r(\theta_1)$ be the unique solution of $f(\theta_1,\r) = 2\delta$.
Then let $\l_0: = 1$ and $\e_0:= 2\delta + 2\r$, 
and let $\l_1$, $\e_1$, $r_1$
be the constants given by the 
above-mentioned~\cite[Theorem~3.1.4]{CDP}.

Given a piecewise geodesic $c$ whose pieces have length 
$\ge r_1$ and meet at angles $\ge\theta_1$, it remains to 
verify that $c$ is an $r_1$--local
$(1,\e_0)$--quasi-geodesic, i.e. for any points $p,q$
on the quasi-geodesic with $d(p,q)\le r_1$ we need to check that
the distance between $p,q$ along $c$ is
$\le d(p,q)+\e_0=d(p,q)+2\delta + 2\r$.
It is enough to consider $p$, $q$ lying 
on consecutive geodesic pieces of $c$; suppose
they form angle $\theta$ at their common vertex $u$.

If either $p$ or $q$, say $q$ is within distance $\r$ of $u$, then
\[
d(p,q) \ge d(p,u) - d(u,q) \ge
d(p,u) + d(u,q) - 2\r,
\] 
as desired, so assume that both $d(p,u)$ and $d(q,u)$ are $\ge \r$.
Consider the geodesic triangle $\Delta = \Delta(p,q,u)$ and a comparison
triangle $\bar\Delta=\bar\Delta(\bar{p},\bar{q},\bar{u})$ in the
hyperbolic plane.
Take $a \in [p,u]$ and $b \in [q,u]$ with $d(a,u)=d(b,u) = \r$,
and let $\bar{a}$ and $\bar{u}$ be comparison points in $\bar\Delta$.
Let $\bar\theta = \angle_{\bar{u}} (\bar{p},\bar{q})$.
Since $\bar\theta \ge \theta \ge \theta_1$, we must have
$d(\bar{a},\bar{b}) = f(\bar\theta,\r) \ge 2\delta$.
By the $\delta$--thinness of $\bar\Delta$,
there exist $\bar{v}$ and $\bar{w}$ on $[\bar{p},\bar{q}]$
such that $d(\bar{a},\bar{v})$ and $d(\bar{b},\bar{w})$ are both at
most $\delta$ and such that
$\bar{w}$ lies between $\bar{v}$ and $\bar{q}$.
Let $v$ and $w$ be the points on $[p,q]$ corresponding to $\bar{v}$
and $\bar{w}$.
Then $d(a,v)$ and $d(b,w)$ are at most $\delta$ and $w$ lies between
$v$ and $q$.
As $d(p,v) \ge d(p,a) - \delta$ and $d(q,w) \ge d(q,b) - \delta$ we get
\[
d(p,q) \ge d(p,v) + d(w,q) \ge d(p,a) + d(q,b) - 2\delta
= d(p,u) + d(u,q) - 2\delta - 2\r,
\] as claimed.
\end{proof}

\small
\bibliographystyle{amsalpha}
\bibliography{rh-qc-revised}
\end{document}